\documentclass[10pt, reqno]{amsart}
\usepackage{amssymb}
\usepackage{amsmath}
\usepackage{amsthm}
\usepackage{amsfonts}
\usepackage{comment}
\usepackage{hyperref}
\usepackage{color}
\usepackage{thmtools}
\usepackage{graphicx}
\usepackage[font=scriptsize]{caption}
\usepackage{appendix}
\usepackage{lipsum}
\usepackage{subcaption}
\usepackage{xr}
\usepackage{geometry}
\usepackage{xspace}
\allowdisplaybreaks
\theoremstyle{plain}

\newtheorem{Thm}{Theorem}
\newtheorem{ImpCor}[Thm]{Corollary}
\newtheorem{Prop}{Proposition}[section]
\newtheorem{Lem}[Prop]{Lemma}

\newtheorem{Cor}[Prop]{Corollary}

\theoremstyle{definition}
\newtheorem{Def}[Prop]{Definition}

\newtheorem{Rem}[Prop]{Remark}

\declaretheoremstyle[
notefont=\bfseries, notebraces={}{},
bodyfont=\normalfont\itshape,
headformat=\NAME\NUMBER\NOTE
]{nopar}

\declaretheoremstyle[
notefont=\bfseries, notebraces={}{},
bodyfont=\normalfont,
headformat=\NAME\NUMBER\NOTE
]{nopardef}

\theoremstyle{nopar}
\newtheorem*{Thm*}{Theorem}
\newtheorem*{Prop*}{Proposition}
\newtheorem*{Lem*}{Lemma}
\newtheorem*{Cor*}{Corollary}
\newtheorem*{ImpCor*}{Corollary}

\graphicspath{{./img/}}

\theoremstyle{nopardef}
\newtheorem*{Def*}{Definition}

\newcommand{\del}{\partial}

\newcommand{\RR}{\mathbb{R}}

\newcommand{\PP}{\mathcal{P}}

\newcommand{\NN}{\mathbb{N}}

\newcommand{\UU}{\mathcal{U}}

\newcommand{\al}{\alpha}

\newcommand{\ga}{\gamma}
\newcommand{\de}{\delta}
\newcommand{\ep}{\epsilon}

\newcommand{\ka}{\kappa}
\newcommand{\la}{\lambda}

\newcommand{\w}{\omega}

\newcommand{\mtxb}{\begin{pmatrix}}
\newcommand{\mtxe}{\end{pmatrix}}
\newcommand{\skewT}{\tilde{T}_{g}}
\newcommand{\XbP}{X\setminus R_P}
\newcommand{\m}{\mathfrak{m}}

\newcommand{\diam}{\operatorname{diam}}
\newcommand{\Lip}{\operatorname{Lip}}
\newcommand{\Height}{\operatorname{Height}}
\newcommand{\Holder}{H\"older\xspace}

\title[Invariant graphs of skew products]{Invariant graphs of a family of non-uniformly expanding skew products over Markov maps}

\author{C. P. Walkden}

\author{T. Withers}\thanks{T. Withers was partially supported by an EPSRC DTA}
\address{School of Mathematics\\ The University of Manchester\\ Oxford Road\\ Manchester\\ M14 9PL\\ UK.}
\email{tom.withers@manchester.ac.uk, charles.walkden@manchester.ac.uk.}

\begin{document}


\begin{abstract}
We consider a family of skew-products of the form $(Tx, g_x(t)) : X \times \RR \to X \times \RR$ where $T$ is a continuous expanding Markov map and $g_x : \RR \to \RR$ is a family of homeomorphisms of $\RR$.  A function $u: X \to \RR$ is said to be an \emph{invariant graph} if $\mathrm{graph}(u) = \{(x,u(x)) \mid x\in X\}$ is an invariant set for the skew-product; equivalently if $u(T(x)) = g_x(u(x))$.  A well-studied problem is to consider the existence, regularity and dimension-theoretic properties of such functions, usually under strong contraction or expansion conditions (in terms of Lyapunov exponents or partial hyperbolicity) in the fibre direction.  Here we consider such problems in a setting where the Lyapunov exponent in the fibre direction is zero on a set of periodic orbits.  We prove that $u$ either has the structure of a `quasi-graph' (or `bony graph') or is as smooth as the dynamics, and we give a criteria for this to happen.
\end{abstract}

\maketitle

\section{Introduction and results}
\subsection{Introduction}

Let $T$ be a continuous, expanding Markov map of the circle,
$X$. Consider the skew product dynamical system $\skewT:X\times \RR\to
X\times\RR$ defined by
\begin{equation}\label{skew product}
\skewT(x,t) = (Tx, g(x,t)).
\end{equation}
where $g(x,t):X\times\RR\to \RR$.  A function $u : X \to \RR$ is said
to be an \emph{invariant graph} if $\mathrm{graph}(u) = \{ (x,u(x))\}$
is $\skewT$-invariant; equivalently
\begin{equation}\label{invariant graph eq}
u(Tx) = g(x,u(x)).
\end{equation}
We refer to $X$ as the \emph{base} and $\RR$ as the \emph{fibre}.  We
are interested in the case when $g(x,\cdot) : \RR \to \RR$ is a
homeomorphism; we normally write $g_x(t)=g(x,t)$, $g_x : \RR \to \RR$
and refer to $g_x(\cdot)$ as a \emph{skewing function}.  When $g_x$ is
uniformly expanding, the invariant graph exists, is \Holder
continuous and, under a partial hyperbolicity assumption, generically has no higher regularity.  As a particular
example, let $b\geq2, b \in \NN$ and let $Tx = bx \bmod 1$.  Let
$\lambda \in (0,1)$, $\lambda b > 1$, and let $g_x(t) = \cos(2\pi x) + \lambda^{-1} t$.
In this case the invariant graph $u(x) = -\sum_{n=0}^{\infty}\lambda^n \cos 2\pi
b^nx$, the classical Weierstrass function.

More generally, if the Lyapunov exponent in the fibre direction is
positive with respect to a given reference measure, then the invariant
graph is measurable, (\ref{invariant graph eq}) holds almost everywhere, and generically is not continuous \cite{Stark2,
  HNW}.  Note that \cite{HNW} requires a partial hyperbolicity
assumption on the skew-product.

In this note, we alter the non-uniform contraction condition
in the fibre and assume that the Lyapunov exponent in the fibre
direction is zero for certain measures.  In particular, we consider
the case when the skewing function is the identity map on a given set
of periodic orbits.  We construct a family of measurable invariant
sets for the skew product and identify the set of measure zero on which (\ref{invariant graph eq}) fails.  Our invariant sets are generically almost
everywhere graphs of functions that are discontinuous on every open
set and are almost everywhere uniformly bounded.

We describe the precise structure of the invariant sets.  The
following dichotomy holds: either the invariant set is of a
discontinuous nature of the form described above or is as smooth as
the dynamics.  The former case is generic; in this case (together with
a partial hyperbolicity assumption) we also calculate the box
dimension of the invariant set in terms of thermodynamic formalism.

The invariant sets we obtain are an example of a family of so-called
bony attractors---a \emph{bony attractor} is a closed set that intersects every almost every
fibre at a single point and any other fibre at an interval. These sets
were first described by \cite{Kudryashov} and other examples occur in
\cite{KleptsynVolk, GharaeiHomburg} as attractors of step
functions over shift maps.

\begin{figure*}[ht]
\centering
\begin{subfigure}[t]{0.5\textwidth}
\centering
\includegraphics[scale = 0.3]{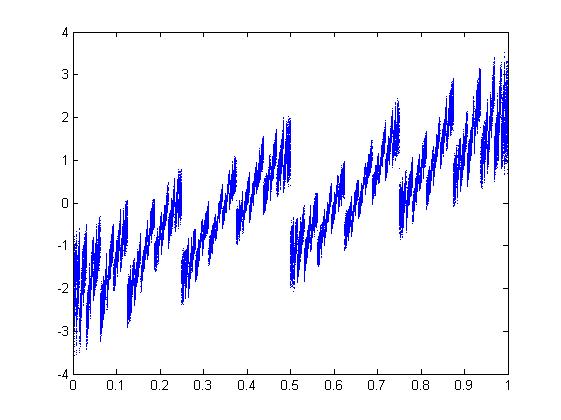}
\centering
\caption{\tiny{An invariant graph defined on $X\setminus R_P$.}}
\label{sin ex qg}
\end{subfigure}\hfill
\begin{subfigure}[t]{0.5\textwidth}
\centering
\includegraphics[scale=0.3]{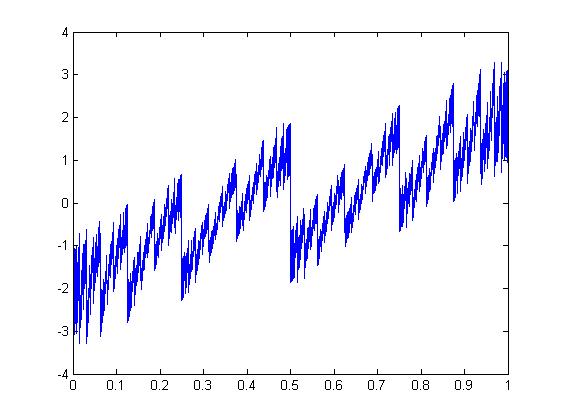}
\centering
\caption{\tiny{The quasi-graph invariant set.}}
\label{sin ex g}
\end{subfigure}
\centering
\caption{The invariant graph and an invariant set (quasi-graph) of the affine real-valued skew product $\skewT(x,t) = (Tx, \sin(2\pi x) + \ga^{-1}(x)t)$, where $\ga(x) =  \frac{3+\cos(2\pi x)}{4}$ and where $T$ is the doubling map. The skew product is the identity at the fixed point $x=0$ and the invariant graph is discontinuous at all pre-images of $0$.}
\label{sin ex}
\end{figure*}

\subsection{Results}
Let $\skewT$ be a skew product as defined in (\ref{skew product}).  We
write $g(x,t)=g_x(t)$ and assume that, for each $x \in X$, $g_x : \RR
\to \RR$ is a homeomorphism.  We also assume that, for each $t\in
\RR$, $x\mapsto g_x(t)$ is $\al$-H\"{o}lder continuous.  We define
$g^n_x(t) = g_{T^{n-1}x}g_{T^{n-2}x}\cdots g_x(t)$ so that
$\skewT^n(x,t) = (T^nx, g^n_x(t))$.  We define $h_x(t) = g_x(t)^{-1}$
so that $g^n_x(t)^{-1} = h^n_x(t) = h_x\cdots h_{T^{n-1}x}(t)$.  Let
$P := \bigcup_{r=0}^{\rho-1} P_r \subset X$ be a collection of $\rho$
distinct periodic orbits for $T:X\to X$.  For $p\in P$ we denote by
$\ell(p)$ the least period of $p$. Let $R_P$ be the set of pre-images
of points in $P$; that is
\begin{equation*}
R_P = \{x\in X\mid T^nx\in P \text{ for some } n\geq 0\}.
\end{equation*}
Note that this is a countable set.  We will often consider the set
$R_P$ and its complement $\XbP$ separately; both sets are
$T$-invariant.

If $h_x : \RR \to \RR$ is a diffeomorphism, we define the
\emph{derivative in the fibre direction} to be
\[
\del h_x(t) = \lim_{\ep\to 0}\frac{h_x(t+\ep) - h_x(t)}{\ep}.
\]
In \S\ref{define skew prod} we will precisely define a set
of skew products $\skewT\in\mathcal{S}(X,\RR, C^\al_P)$ where $\skewT$
is expanding in the fibre direction except along the periodic orbits
in $P$ where $g^{\ell(p)}_p(t)=t$ for all $t\in\RR$.  We shall abuse notation
slightly and write $g\in C^\al_P$.

\begin{Rem}\label{diffble remark}
Examples of skew products that satisfy our hypotheses include affine maps $g_x(t) = f(x)+\ga(x)^{-1}t$ where $f$, $\ga$ are $\al$-H\"{o}lder with $\ga(p)=1$, $f(p)=0$ for $p\in P$ and otherwise $0<\ga(x)<1$. Our conditions remain satisfied for diffeomorphisms $g:X\to \text{Diff}(\RR,\RR)$ defined by small, sufficiently smooth perturbations of affine maps preserving conditions \eqref{non expanding} to \eqref{g holder} below. Figure \ref{sin ex} shows an explicit example.
\end{Rem}

First, we show that such invariant graphs exist. We prove that we have a unique invariant
function $u$ on $\XbP$; hence the graph of this function is an
invariant set of the restricted skew product
$\skewT|_{\XbP}$.

\begin{Thm}\label{existence and uniqueness of la1}
Let $\skewT\in \mathcal{S}(X,\RR, C^\al_P)$. There exists a uniformly bounded function $u:X\setminus R_P\to \RR$ such that $u(Tx)=g_x(u(x))$ for $x\in X\setminus R_P$.  Any other uniformly bounded function $v:X\setminus R_P\to \RR$ satisfying this equation is equal to $u$; moreover, if $\mu$ is an ergodic measure for $T$ not supported on $R_P$, then $u$ is $\mu$-a.e.\ unique amongst the set of measurable functions.
\end{Thm}

We prove the following corollary to Theorem \ref{existence and uniqueness of la1}, showing that $u$ is uniquely defined on $\XbP$ but can be arbitrarily defined on $R_P$.
\begin{ImpCor}\label{existence of the us}
Let $\skewT\in \mathcal{S}(X,\RR, C^\al_P)$ and let $u$ be as in
Theorem~\ref{existence and uniqueness of la1}. Suppose that $P$
consists of $\rho$ periodic orbits. For each $r$ such that $0\leq r
\leq \rho-1$, choose one element of each periodic orbit $p\in
P_r\subset P$.  There is a $\rho$-parameter family of
invariant graphs for the skew product, $u_s:X\to\RR$, where
$u_s(p)=s_r$ for $s=(s_0,\dots, s_{\rho-1})\in \RR^\rho$. Furthermore,
$u_s|_{\XbP} = u$.
\end{ImpCor}

Clearly, for $x \in R_P$, each fibre $\{x\}\times\RR$ is $\skewT$-invariant; by taking the union we have a dense $\skewT$-invariant set. As $\XbP$ is dense in $X$, we can consider a sequence $x_0\to x$ for any $x\in R_P$ where each point $x_0\in \XbP$.

\begin{Def}\label{quasi def}
Let $\skewT\in \mathcal{S}(X,\RR, C^\al_P)$. Define the \emph{invariant quasi-graph} $U\subset X\times \RR$ as $U=\bigcup_{x\in X} U_x$ where
\[
U_x =
\begin{cases} (x,u(x)) &\text{ if } x \in \XbP,\\
			\{x\}\times [\liminf_{y \in \XbP, y\to x}u(y),\limsup_{y\in\XbP, y\to x}u(y)] &\text{ if } x\in R_P.
\end{cases}
\]
\end{Def}
In other words, we connect the discontinuities between the values of the function $u$ as we approach $R_P$; see Figure~\ref{sin ex}. Denote the length of the interval of the quasi-graph at $x\in R_P$ by
\[
|U_x| = \left|\limsup_{y \in \XbP, y\to x}u(y)-\liminf_{y \in\XbP, y\to x}u(y)\right|;
\]
if $x\in \XbP$ then we set $|U_x|=0$. In Proposition \ref{quasi-graph inv} we show that $U$ is a $\skewT$-invariant set.

We prove two main results about quasi-graphs, reminiscent of those found in the studies of Weierstrass functions, \cite{Baranski, HuLau} and for dynamically-defined invariant graphs \cite{Stark2, HNW}. The first is a dichotomy of the structure of the invariant graphs.

\begin{Thm}\label{dichotomy thm}
Let $\skewT\in\mathcal{S}(X,\RR, C^\al_P)$. Then either:
\begin{enumerate}
\item the invariant quasi-graph $U$ is the graph of a uniformly $\al$-H\"{o}lder continuous function;
\item for every $x\in R_P$, $|U_x|>0$.
\end{enumerate}
\end{Thm}

When we are in the second case, we will often say ``the quasi-graph
$U$ is not the graph of an invariant function". As $R_P$ is dense in
$X$, we see that $U$ is not the graph of a function on any open set
and has a `vertical jump' at each element of $R_P$.  It is easy to
construct examples of continuous invariant graphs, but our second
result proves that generically $U$ is of the discontinuous type.

\begin{Thm}\label{open dense f}
Let $\skewT\in\mathcal{S}(X,\RR, C^\al_P)$. There exists a $C^{\al}$-open and $C^0$-dense set of $g\in C^\al_P$ such that the invariant quasi-graph $U$ of $\skewT$ is not the graph of a function.
\end{Thm}

For higher regularity of $u$, we need higher regularity on the base dynamics $T$ and we
consider $C^{r+\al}$-expanding
endomorphisms of the circle, with $r\in\NN$, $r\geq 1$ and $\al>
0$. We also assume that that $g_x : \RR \to \RR$ is $C^{r+\al}$ and,
for each $t$, $x \mapsto g_x(t)$ is $C^{r+\al}$.  Again, we abuse
notation slightly and write $g\in C^{r+\al}_P$. Denote this set of
skew products by $\skewT \in
\mathcal{S}_{C^{r+\al}}(X,\RR,C^{r+\al}_P)$.  In this case we can
strengthen the above dichotomy: if $U$ is the graph of a function then
it is as smooth as the dynamics.
\begin{Thm}\label{dichotomy thm strong}
Let $\al\geq 0$, $r\in\NN$ and $r\geq 1$. Let $\skewT\in \mathcal{S}_{C^{r+\al}}(X,\RR, C^{r+\al}_P)$. Then either:
\begin{enumerate}
\item the invariant quasi-graph $U$ is the graph of a $C^{r+\al}$ function;
\item for every $x\in R_P$, $|U_x|>0$.
\end{enumerate}
\end{Thm}
Our final result concerns the box dimension of the invariant
quasi-graphs. Our result extends that in  \cite{Bedford} to our
setting, and we believe it is the first attempt to calculate the
dimension of invariant graphs of non-uniformly expanding skew
products. The key difficulty is establishing a sufficiently strong
form of bounded distortion.

Let $q:X\to \RR$ be a function and let $\m(q)=\inf_{x\in X}q(x)$. Suppose $\del g_x(t)>0$; we will show in Remark \ref{g orient pres} that this causes no loss in generality. A skew product is \emph{partially hyperbolic} if there exists $\ka>1$ such that
\begin{equation}\label{partial hyp}
1<\ka\leq \m(\del h)\m(|T'|)
\end{equation}
where the infimum $\m(\del h)$ is taken over all $(x,t)\in X\times\RR$
and $h_x=g_x^{-1}$.

Suppose that $g:X\to C^{2}(\RR,\RR)$ is a $C^2$ diffeomorphism and
we return to allowing $T:X\to X$ to be a continuous expanding Markov map (with conditions specified in \S\ref{define skew prod}). Let $x\mapsto h_xt$, $x\mapsto \del h_x(t)$ and $x\mapsto \del^2 h_x(t)$ be Lipschitz continuous. Denote this set of skew products by $\mathcal{S}(X,\RR, C^{\Lip,2}_P)\subset\mathcal{S}(X,\RR,C^\al_P)$.

Define the function
\[
\mathcal{D} h(x) = \begin{cases} \del h_x(u(Tx)) &\text{ for } x\in
  \XbP,\\ \del h_x(t) &\text{ for } x\in R_P, \text{ where } t=\limsup_{y \in \XbP, y\to Tx}u(y).\end{cases}
\]
Let $\PP(\phi)$ be the topological pressure of a function $\phi:X\to \RR$, see Definition \ref{pressure}. We prove the following.
\begin{Thm}\label{box dim thm}
Let $\tilde T_{g}\in \mathcal{S}(X,\RR, C^{\Lip,2}_P)$ be partially hyperbolic such that the invariant quasi-graph is not the graph of a continuous function on $X$.  The box dimension of the quasi-graph $U$ is the unique solution $t$ to the generalised Bowen equation
\begin{equation}\label{bowen eq}
\PP((1-t)\log |T'| +\log |\mathcal{D} h|)=0.
\end{equation}
\end{Thm}
As $\del h_x(t)$, and so $\mathcal{D}h(x)$, are uniformly bounded for $(x,t)\in X\times\RR$, it is well-known that a unique solution to such an equation for a partially-hyperbolic skew-products exists and often corresponds to the dimension of graphs \cite{Bedford, MossWalkdenDim}.

\subsection{The family of skew products}\label{define skew prod}
Here we list the technical hypotheses on the dynamics.  We assume that $T$ is a $C^1$ expanding Markov map.  Specifically, there is a
partition $X = \bigcup_{j=0}^{b-1} X_j$, $X_j = [t_j, t_{j+1}]$ with $0 = t_0 < t_1 < \cdots < t_b=1$, such that, for each $j$, $T|_{X_j}:\mathop{\mathrm{Int}}X_j\to
\bigcup_{i\in I_j}X_i$ is a $C^1$ diffeomorphism.  We assume that $T$ is continuous.  We assume there exists $\theta<1$ such that $|T'(x)|\geq\theta^{-1}>1$ for all $x\in
\bigcup_{i=0}^{b-1}\text{Int}(X_j)$.  We assume that $T$ is locally eventually onto, namely that there exists $N\in\NN$ such that, for all
$X_j$, $T^NX_j=X$ (equivalently, $T^N$ is full branched for some $N\in\NN$).  As $T|_{X_j}$ is a diffeomorphism onto its image, there
exists a well-defined inverse branch $\w_j:\bigcup_{i\in I_j}X_i\to X_j$ and $\w_j$ is a diffeomorphism. Note that
$d(\w_jx,\w_jy)\leq \theta d(x,y)$. Define a cylinder of rank $n$ by $C_n = C_{j_0j_1\dots j_{n-1}} = \w_{j_0}\w_{j_1}\dots\w_{j_{n-1}}(X)$. Note that $C_n \subset X$ is an interval.

We now state the hypotheses on the skewing function.  Let $P =
\bigcup_{r=0}^{\rho-1}P_r$ be a finite set of periodic orbits.  If $p \in
P_r$ then we write $\ell(p)$ for the least period of $p$.  Recall
$h_x(t)=g_x(t)^{-1}$.

Fix a constant $C_\mathcal{S}\geq 1$. Suppose the skewing function $g$ is such that
\begin{equation}\label{indentity under orbit}
g^{\ell(p)}_{p}(t)=t,
\end{equation} for all $t\in\RR$ and $p\in P$. Suppose that
\begin{equation}\label{non expanding}
|h^{n}_x(t)-h^{n}_x(t')|\leq C_{\mathcal{S}}|t-t'|
\end{equation}
for all $x\in X$. Let $\ep>0$ be small. Define the collection of open balls (intervals) in $X$ of radius $\ep$ centred on $p\in P$ by  $B_\ep(P) = \bigcup_{p\in P} B_\ep(p)$ and let $G=X\backslash B_\ep(P)$.
 For an orbit segment $x,\dots, T^{n-1}x\in B_\ep(P)$, denote $\ep_j = \inf_{p\in P}d(T^j x, p)$. Suppose we can define a function $\la:[0,\ep]\to \RR^+$ where $\log\la$ is $\al$-H\"{o}lder for some $\al>0$, with $\la(\ep)>0$ and $\la(0)=1$ such that
\begin{equation}\label{slow contracting}
|h^{n}_x(t)-h^{n}_x(t')|\geq C_{\mathcal{S}}^{-1}\la(\ep_0)\dots\la(\ep_{n-1})|t-t'|.
\end{equation}
Let $\de$ be such that $0<\de\leq \ep$. Suppose that the orbit segment $x,\dots, T^{n-1}x$ visits $X\backslash B_\de(P)$ $j$-many times, let there exist $0<\la_\de<1$ such that
\begin{equation}\label{fast contracting}
|h^{n}_x(t)-h^{n}_x(t')|\leq C_{\mathcal{S}}\la_\de^j|t-t'|,
\end{equation}
where $\la_\de$ is independent of $x$, $t$ and $t'$. We also assume that both $x\mapsto g_x(t)$ and $x\mapsto h_x(t)$ are $\al$-H\"{o}lder continuous,
\begin{equation}\label{g holder}
|g_x(t)- g_y(t)|\leq C_g(t)d(x,y)^\al \hspace{.5cm}\text{ and }\hspace{.5cm}
|h_x(t)- h_y(t)|\leq C_{h}(t)d(x,y)^\al,
\end{equation}
for some $\al>0$. If $t\in W$ and $W\subset \RR$ is compact, then define $C_W=\sup_{t\in W}\{C_g(t), C_{h}(t)\}$.

As $x\mapsto g_x(t)$ is continuous and $X$ is compact, there exists $K_g>0$ independent of $x$ such that
\begin{equation}\label{g bounded}
|g^n_x(t)|\leq \sup_{0\leq j\leq n} \sup_{x\in X}|g^n_x(t)|= K_g(t,n).
\end{equation}
Finally, as $g$ is invertible and continuous, there exists $\la_{\min}>0$ such that
\begin{equation}\label{g non0}
|h^{n}_x(t)-h^{n}_x(t')|\geq C^{-1}_{\mathcal{S}}\la_{\min}^n|t-t'|.
\end{equation}
Let $T:X\to X$ be a continuous, expanding, locally eventually onto Markov map. We denote the set of skew products of the form \eqref{skew product} satisfying conditions \eqref{indentity under orbit} to \eqref{g holder} with fixed constant $C_\mathcal{S}$ by $\mathcal{S}(X,\RR, C^\al_P)$.

\begin{Rem}\label{examples and constant}
Suppose $g_x(t) = f(x)+\ga^{-1}(x)t$ as in Remark \ref{diffble remark}, where $\sum_{i=0}^{\ell(p)}f(T^ip)=0$, $0<\ga(x)\leq 1$ and $\ga(p)=1$ if and only if $p\in P$. The function $\la:[0,\ep]\to \RR^+$ in condition \eqref{slow contracting} can be defined as $\la(\ep_x) = |\ga(x)|$ for $x\in B_\ep(P)$ where $\ep_x = \inf_{p\in P}d(x,p)$. In this case, for $\de\leq \ep$, $\la_\de$ in condition \eqref{fast contracting} is chosen as $\sup_{x\not\in B_\de(P)}|\ga(x)|<1$; also $\la_{\min}=\m(|\ga|)$.

For these skew products, the constant $C_\mathcal{S}=1$. Allowing the constant to be larger than $1$ allows $|\ga(x)|>1$ at some $x\in X$ provided that $|\ga^n(x)|<C_{\mathcal{S}}$ for all $n$, hence the Lyapunov exponent is non-positive. Throughout, we fix the constant $C_\mathcal{S}$ as, when we make a small perturbation of the skew product, the constant does not necessarily perturb independent of $n$. This is important in the proof of Theorem \ref{open dense f}.

More generally, for $g_x$ a diffeomorphism with $0<\la_{\min}\leq |\del h_x(t)|<1$ for $x\not\in P$ and $h^{\ell(p)}_p(t)=t$ if and only if $p\in P$, we define $\la(\ep_x) = \m(|\del h_x|)$ and $\la_\de = \sup_{x\not\in B_\de(P)}\|\del h_x\|_\infty$.
\end{Rem}

\begin{Rem}
Notice, unlike \cite{Bedford} and \cite{HNW} we have no partial hyperbolicity condition in general.
\end{Rem}

\begin{Rem}
Suppose that the skewing function is not the identity over the periodic orbits of $P$ but has zero Lyapunov exponent. Say $g_x(t) = f(x)+ \ga(x)^{-1}t$, where $\ga^\ell(p)=1$ but $f(p)>0$ for $p\in P$. In this setting, it is easy to see that the invariant graph is unbounded on pre-images of $P$. In fact, the two basins (of points repelled to $\pm\infty$) appear intermingled in the sense of \cite{AYYK, AshwinPodvigina, KellerSD}.
\end{Rem}

\begin{Rem}
For Theorem \ref{existence and uniqueness of la1} and Corollary \ref{existence of the us} we do not need $T$ to be locally eventually onto, rather just continuous and expanding.
\end{Rem}

\section{Properties of Markov maps}

To prove Theorem \ref{existence and uniqueness of la1}, we will need
the following technical lemma.  The content of the lemma is surely well-known, but we include a proof for completeness.
\begin{Lem}\label{aux lem bad set}
Let $T:X\to X$ be a continuous expanding Markov map of the circle.
Let $p$ be a periodic point and let $P$ denote the periodic orbit of
$p$.  There exists $\ep>0$ such that for all $\de$ where $0<\de\leq
\ep$ if the orbit sequence $\{x, Tx,\dots, T^nx\}\subset B_\de(P)$
with $x\in B_\de(p)$, then $d(T^jx,T^jp)<\de\theta^{n-j}$ for all
$0\leq j\leq n$.
\begin{proof}
If $x=p$, then, as $P$ is $T$-invariant, the result is trivial.

We first prove that if $x\in B_\de(p)$ and $Tx\in B_\de(P)$, then $Tx\in B_\de(Tp)$. Let $t_j$ be the lower end-point of the interval $X_j$. Let $Q = P\cup \{t_i\}_{i=0}^{b-1}$. Let $d = \inf_{q_i\not = q_j\in Q}d(q_i,q_j)$.  Choose $\ep < d/\|T'\|_\infty 4$ and let $0 < \de \leq \ep$.
As $\|T\|_\infty>1$, $\de<d/4$ and so for $p\in P$ the balls $B_\de(p)$ are pairwise disjoint.

For $x\not=p$, $x$ is located in precisely one of $(p-\de,p)$ or $(p,p+\de)$ and on these sets $T$ is differentiable (note, $p$ could be equal to $t_j$ for some $j$, hence we split the interval $B_\de(p)$ at $p$). By the Mean Value Theorem, $d(Tx, Tp)\leq\|T'\|_\infty d(x,p)<d/4$.
Hence the set $T(B_\de(p))$ has diameter at most $d/2$ and
\begin{equation}\label{ball contained in ball}
T(B_\de(p))\subset B_{d/2}(Tp).
\end{equation}
By assumption, $Tx\in B_\de(P)$. We show that $Tx\in B_\de(Tp)$. Suppose not, say $Tx\in B_\de(p')$ with $p'\in P$, $p'\not= Tp$. Then $Tx\in B_\de(p')\cap T(B_\de(p))$. By \eqref{ball contained in ball}, $Tx\in B_{d/2}(p')\cap B_{d/2}(Tp)$, contradicting that the points of $P$ are at least distance $d$ apart. So $Tx\in B_\de(Tp)$.

Suppose that $p\in\text{Int}(X_j)$ for some $j$; therefore $\w_j(Tp)=p$. By choice of $\de$, $T:B_\de(x)\to X$ is a diffeomorphism on $B_\de(x)\subset X_j$. Therefore, the inverse branch $\w_j:T(B_\de(x))\to X_j$ is a diffeomorphism on $T(B_\de(x))$. As $T$ is expanding, $T(B_\de(p))\supset B_\de(Tp)$. By the Mean Value Theorem, for $x\in B_\de(p)$,
\begin{align}\label{techy lemma mvt}
\theta\geq \frac{d(\w_j(Tx),\w_j(Tp))}{d(Tx,Tp)}= \frac{d(x,p)}{d(Tx,Tp)}.
\end{align}
As $Tx\in B_\de(Tp)$, we have $d(Tx,Tp)<\de$ and so $d(x,p)<\theta\de$.

Finally, consider the case where $p$ is the unique point of intersection $p \in X_{j-1}\cap X_{j}$. Then $\w_j(Tp)=p=\w_{j-1}(Tp)$. Either $x\in (p-\ep, p)\subset X_{j-1}$ or $x\in (p,p+\ep)\subset X_j$. In either case, by applying the appropriate inverse branch, the result follows by the idea of \eqref{techy lemma mvt} above, replacing the ball $B_\de(p)$ with intervals $(p,p+\ep)$ or $(p-\ep,p)$. By iterating this argument, if $T^jx\in B_\de(P)$ for all $0\leq j\leq n$, then $d(T^jx,T^jp)<\de\theta^{n-j}$, as required.
\end{proof}
\end{Lem}

The following corollary is also well-known: the only orbit that $\de$-shadows a periodic orbit (for $\de$ sufficiently small) is the periodic orbit itself.
\begin{Cor}\label{cor to aux lem}
Let $\de>0$ be as in Lemma \ref{aux lem bad set}. Suppose for all $n\in\NN$ we have $T^nx\in B_\de(P)$.  Then $x\in P$.
\begin{proof}
Suppose $x\not\in P$.  There exists $\eta>0$ such that, for all $p\in P$, $d(x,p)>\eta$. As $T^nx\in B_\de(P)$ for all $n\in\NN$, for any $m\in \NN$ the orbit segment $\{x,\dots, T^mx\}\subset B_\de(P)$. By Lemma \ref{aux lem bad set}, for some $p\in P$, we have $d(x,p)< \de\theta^m$.  Choose $m$ large such that $\de\theta^m<\eta$, contradicting our assumption that $x\not\in P$.
\end{proof}
\end{Cor}

In order to prove the existence of an invariant graph, we show that $h^{n}_x(t)\to u(x)$. In our setting, when $x$ is very close to $P$, the rate of contraction of $h_x(t)$ is not uniformly bounded below $1$. Thus, supposing that the orbit segment $Tx,\dots, T^{n-1}x$ remains in $B_\de(P)$, the next result uses Lemma \ref{aux lem bad set} to bound $h^{n}_x(t)$ independently of $x$ and $n$.

\begin{Lem}\label{proving f bounded on bad orbits}
Let $T:X\to X$ be a continuous expanding Markov Map of the
interval. Let $0<\de\leq\ep<1$ be as in Lemma \ref{aux lem bad
  set}. Let $g\in C^\al_P$. Let $n\geq 2$. If the orbit segment
$\{Tx,\dots, T^{n-1}x\}$ lies in $B_\de(P)$, then $|h^{n}_x(t)-t|\leq
A_h(t)$, where $A_h(t)>0$ is independent of $n$ and $x$, but not $t$.
\begin{proof}
Let $\ell$ be the least period of $p\in P$. Let $Tx\in B_\de(Tp)$. Let $n_\ell$ be the smallest integer $n_\ell\geq n$ divisible by $\ell$, so $h^{n_\ell}_p(t)=(t)$ and $n_\ell-n \leq \ell-1$

As $\{Tx, \dots, T^{n-1}x\}\subset B_\de(P)$, by Lemma \ref{aux lem
  bad set}, for $0\leq i\leq n-2$, $d(T^i(Tx), T^i(Tp))\leq
\theta^{n-2-i}\de$.  Let $T^{n_\ell-n}q=p$. As $h^{n_\ell}_q$ is the
identity,
\begin{align}\label{bound to show f bound}
|h^{n}_x(t)-t| &= |h^{n}_x(t)- h^{n_\ell}_q(t)| \leq |h^{n}_x(t)-h^{n}_p(t)|+|h^{n}_p(t)-h^{n}_p(h^{n_\ell-n}_q(t))|.
\end{align}
We bound the first term of \eqref{bound to show f bound}. By \eqref{non expanding} and Lemma \ref{aux lem bad set},
\begin{align*}
\nonumber |h^{n}_x(t)- h^{n}_p(t)|&\leq \sum_{i=0}^{n-1}|h^{i}_xh_{T^ix}h^{n-1-i}_{T^{i+1}p}(t)-h^{i}_xh_{T^ip}h^{n-1-i}_{T^{i+1}p}(t)|\\
\nonumber &\leq C_{\mathcal{S}}C_g(h^{n-1-i}_{T^{i+1}p}(t))\sum_{i=0}^{n-1}d(T^ix,T^ip)^{\al}\\
&\leq C_{\mathcal{S}}C_g(h^{n-1-i}_{T^{i+1}p}(t))\left(1+\sum_{i=0}^{n-2}\theta^{\al(n-2-i)}\de\right) \leq C(\ell, t).
\end{align*}
as $h^{n-1-i}_{T^{i+1}p}(t)\in [-K_g(\ell, t),K_g(\ell, t)]=:W$, the H\"{o}lder constant $C_g(h^{n-1-i}_{T^{i+1}p}(t))<C_W$. We bound the second term of \eqref{bound to show f bound} by
\begin{align*}
|h^{n}_p(t)-h^{n}_p(h^{n_\ell-n}_q(t))|&\leq C_{\mathcal{S}}|t-h^{n_\ell-n}_q(t)|\leq C_{\mathcal{S}}(|t|+K_g(t,\ell))
\end{align*}
as $n_\ell-n\leq \ell$. Hence, we have the bound as required.
\end{proof}
\end{Lem}

\section{Existence of the invariant graph}
We consider the existence and uniqueness of the invariant graph and
prove Theorem \ref{existence and uniqueness of la1} and Corollary
\ref{existence of the us}.
\begin{proof}[Proof of Theorem \ref{existence and uniqueness of la1}]
Let $\de$ be sufficiently small so that Lemma \ref{aux lem bad set} holds. Define $u_n:\XbP\times\RR\to \RR$ by
$u_n(x,t) = h^{n}_x(t)$.
By Corollary \ref{cor to aux lem}, the set of points with orbit that visits $X\setminus B_\ep(P)$ infinitely often are precisely $\XbP$. Denote the set $X\setminus B_\de(P) = G$.

Suppose that the orbit segment $T^{n_{i-1}+1}x, \dots, T^{n_i-1}x\in B_\de(P)$ and $T^{n_i}x\in G$. If $x\not\in G$, we let $n_{-1}=0$ and $n_0$ is the first visit to $G$, and if $x\in G$ then $n_0=0$. In both cases $n_1$ is the first return to $G$. Define
\[
H_{T^{n_{i-1}}x} = h^{n_i-n_{i-1}}_{T^{n_{i-1}}x}.
\]
Let $j$ be the number of times the orbit segment $x,\dots, T^nx$ visits $G$.  We have
\[
h^{n}_x(t) = H^{j}_x(t) = H_xH_{T^{n_1}x}\dots H_{T^{n_{j-2}}x}H_{T^{n_{j-1}}x}(t).
\]
By Lemma \ref{proving f bounded on bad orbits}, $|H_{T^{n_{i-1}}x}(t)- t|\leq A_h(t)$, where $A_h$ is independent of $n, i$ and $x$ but depends on $t$. Fix $n,m\in\NN$. Let $n_j\leq n<m\leq n_k$ where $T^{n_j}x$ is the largest $n_j\leq n$ such that $T^{n_j}x\in G$ and $n_k$ is the smallest $n_k\geq m$ such that $T^{n_k}x\in G$. Fixing $t\in\RR$, we bound
\[
|h^{n}_x(t)-h^{m}_x(t)| = |H^{j}_x(h^{n-n_j}_{T^{n_j}x}(t))-H^{j}_x(h^{m-n_j}_{T^{n_j}x}(t))| \leq \la_\de^j|h^{n-n_j}_{T^{n_j}x}(t)-h^{m-n_j}_{T^{n_j}x}(t)|.
\]
As $n,m\to\infty$, we have $\la_\de^j\to 0$. By repeatedly applying Lemma \ref{proving f bounded on bad orbits},
\begin{align}\label{existence bound}
\nonumber |h^{n-n_j}_{T^{n_j}x}(t)-h^{m-n_j}_{T^{n_j}x}(t)| &= |h^{n-n_j}_{T^{n_j}x}(t)-H^{k-1-j}_{T^{n_j}x}(h^{m-n_{k-1}}_{T^{n_{k-1}}x}(t))|\\
&\nonumber\leq |h^{n-n_j}_{T^{n_j}x}(t) - t|+|t-H_{T^{n_j}x}(t)|+|H_{T^{n_j}x}(t)-H^{2}_{T^{n_j}x}(t)|+\\
&\nonumber\hspace{2cm}\dots + |H^{k-j-1}_{T^{n_j}x}(t)- H^{k-j-1}_{T^{n_j}x}(h^{m-n_{k-1}}_{T^{n_{k-1}}x}(t))|\\
&\leq A_h(t)+\sum_{i=0}^{k-j-1}\la_\de^i A_h(t)\leq C(t)
\end{align}
for some $C(t)>0$. Therefore, for each $x,t$, the sequence $h^{n}_x(t)$ is Cauchy. By \eqref{fast contracting}, it follows that the limit is independent of $t$. We can define a uniformly bounded (on $\XbP$) function $u$ by
$u(x) = \lim_{n\to\infty}u_n(x,t)=\lim_{j\to\infty}H^{j}_x(t)$.
Then, $u$ is an invariant graph on $\XbP$ as
$u(Tx)=\lim_{n\to\infty}h^{n}_{Tx}(t)= g_x\left(\lim_{n\to\infty}h^{n}_x(t)\right) = g_x(u(x))$.
To prove uniqueness, suppose $v$ is another invariant graph and suppose that $v$ is uniformly bounded on $\XbP$.  For $x\in\XbP$, and all $n\in\NN$, we have $v(x) = h^{n}_x(v(T^nx))$. So,
\[
h^{n}_x(-\|v\|_\infty)\leq v(x) \leq h^{n}_x(\|v\|_\infty).
\]
As $h^{n}_x(t)$ converges to $u(x)$ as $n\to\infty$ and the limit is independent of $t$, $v(x)=u(x)$ for $x\in\XbP$.

Now, suppose that $v$ is only measurable. Let $B>0$, Suppose $W=\{x\mid |u(x)|<B\}$.  Let $\mu$ be an ergodic measure for $T$ that is not supported on $R_P$.  For $B$ sufficiently large, $\mu(W)>0$. By the ergodicity of $T$, for $\mu$-a.e. $x$ there exists subsequence $n_m$ such that $T^{n_m}x\in W$. Hence,
\[
h^{n_m}_x(-B)\leq u(x) \leq h^{n_m}_x(B).
\]
As $m\to\infty$, we have $h^{n_m}_x(t)\to u(x)$ independently of $t$.
Hence $v=u$ $\mu$-a.e.
\end{proof}

While the bound on $u$ is uniform in $x$, the rate of convergence is not uniform; this lack of uniform
convergence gives the invariant graph its interesting non-continuous
structure.

\subsection{Proof of Corollary \ref{existence of the us}}
We now consider invariant graphs when the dynamics is restricted to
the set $R_P$ and prove Corollary \ref{existence of the us}.

\begin{proof}[Proof of Corollary \ref{existence of the us}]
Let $P_r$ be a periodic orbit in $P$ and, for each $r$, choose $p_r
\in P_r$.  Recall $\ell(p_r)$ denotes the least period of $p_r$.  As
$u$ is assumed to be $\skewT$-invariant on $R_P$, we have
$g_{p_r}(u(p_r))=u(Tp_r)$.  Iterating this we have
$g^\ell_{p_r}(u(p_r)) = u(T^{\ell(p_r)} p_r)=u(p_r)$.  As
$g^{\ell(p_r)}_{p}$ is the identity, any value of $u(p_r)$ satisfies
this equation.

Choose (arbitrarily) $s=(s_0,\dots, s_{\rho-1})\in \RR^\rho$.  For our
choice of $p_r$, define $u_s(p_r) = s_r$.  As $u_s(Tx)=g_x(u(x))$,
this then defines $u_s$ on $P$.  Note that $u_s (\w_j p)) = h_{\w_j
  p}u_s(p)$; we can iterate this to define
$u_s$ on $R_P$.

Define $u_s=u$ on $\XbP$.  It remains to show that $u_s$ has a bound
depending only on $s$ and $u:\XbP\to\RR$. We only need to consider
points in $R_P$ as $u_s(x)=u(x)$ for $x\in\XbP$. Let $x\in R_P$, so
for some $N\in\NN$ we have $T^Nx\in P$, say $T^Nx=p$. We have
$g^N_{x}(u_s(x))= u_s(p)$.
As $u_s(p)=s_r$ it suffices to bound $h^{N}_x(s_r)$ for $s_r\in\RR$. If $x\in P$, then $x$ is a periodic point and, by \eqref{g bounded}, $|g^N_p(s_r)|\leq K_g(N,s_r)\leq K_g(\ell,s_r)$. Otherwise, let $J$ be the total number of visits of the orbit of $x$ to $G$, denoting each visit $T^{N_i}x\in G$ for $0\leq i<J-1$ and $N=N_J$ so that $T^{N_J}x=p$. Again, if $x\not\in G$, let $N_{-1}=0$ so that $N_0$ is the first visit to $G$.

Using the notation of the proof of Theorem \ref{existence and uniqueness of la1}, we write $h^{N}_x(s_r) = H^{J-1}_xh^{N-N_{J-1}}_{T^{N_{J-1}}x}(s_r)$. By the same arguments as \eqref{existence bound}, we can bound $|H^{J-1}_x(h^{N-N_{J-1}}_{T^{N_{J-1}}x}(s_r)) - s_r|<C(s_r)$ for some $C(s_r)>0$. Therefore $h^{N}_x(s_r)$ is bounded independently of $N$.
\end{proof}

\subsection{Reducing to fixed points}
Having shown that the invariant sets described in Theorem
\ref{existence and uniqueness of la1} and Corollary \ref{existence of
  the us} exist for $\skewT$, it will be useful in what follows to
replace $T$ by a power and assume that $P$ consists of fixed points
such that $g_p(t)=t$ for all $p\in P, t \in \RR$.  Moreover, we can
also assume that $T$ is full branched.  The following allows us to do
this.


\begin{Prop}\label{fixed pts ok}
Let $P$ be a set of periodic orbits of $X$. Let the skew product $\skewT\in \mathcal{S}(X,\RR,C^\al_P)$ and let $N$ be the least integer such that $T^NX_j=X$ for all $0\leq j\leq b-1$. Let $\ell=k\ell_p$ where $k\in\NN$ and $\ell_p$ is the lowest common multiple of the periods of orbits in $P$ and $N$. Let $R_P$ be the set of pre-images of $P$ under $\skewT$ and $R_P^\ell$ be the set of pre-images of $P$ under $\skewT^\ell$. Then $R_P=R^\ell_P$, the skew product $\skewT^\ell|_{X\setminus R^\ell_P}$ has unique, uniformly bounded invariant graph $u:X\setminus R^\ell_P\to \RR$ and the graph is equal to to the unique uniformly bounded invariant graph of the skew product $\skewT|_{\XbP}$.

\begin{proof}
By our choice of $\ell$, the set $P$ consists of points that are fixed under $T^\ell:X\to X$. Let $x\in R_P$. So, there exists $n_0\in\NN$ such that for $n\geq n_0$ we have $T^nx\in P$. Choose $M>n_0$ the smallest integer such that $M=\ell K$ for $K\in \NN$. Let $k\geq K$. Then $T^{\ell k}x \in P$, so $x\in R^\ell_P$. For the other direction, if $x\in R^\ell_P$ then there exists $K\in\NN$ such that, for all $k\geq K$, $T^{\ell k}x \in P$. As $P$ is $T$ invariant, $T^nx\in P$ for all $n\geq \ell K$, therefore $x\in R_P$. Thus, $R_P=R^\ell_P$.

Let $g^\ell_x = \ga_x.$ Then, $\skewT^\ell(x,t)= (T^\ell x, \ga_x(t))$ is contained in $\mathcal{S}(X,\RR,C^\al_P)$, hence satisfies the conditions of Theorem \ref{existence and uniqueness of la1} and so there exists a unique, bounded invariant graph of $\tilde T_{\ga}$ say $u':\XbP\to \RR$.

Let $u:\XbP\to\RR$ be the unique invariant graph of $\skewT|_{\XbP}$. Clearly, $u$ is also an invariant graph of $\tilde T_{\ga}|_{X\setminus R^\ell_P}$. As $R_P=R^\ell_P$ and the invariant graph is unique, $u'=u$.
\end{proof}
\end{Prop}

From now on we will assume without loss of generality that the skew
product is such that $T$ is full branched and $P$ consists only of
fixed points.

\begin{Rem}\label{g orient pres}
Proposition \ref{fixed pts ok} also allows us to assume $g$ is orientation preserving. As each $g_x:\RR\to\RR$ is invertible, $g^2_x$ is orientation preserving and by Proposition \ref{fixed pts ok} the skew product $\tilde T_{g^2}$ has the same invariant graph as $\skewT$.  (Note that we cannot assume that $T$ is orientation preserving: consider the tent map.)
\end{Rem}

\section{The invariant quasi-graph}
\subsection{Structure of the invariant quasi-graph}
Recall Definition \ref{quasi def} of the quasi-graph $U$. As the invariant graph $u$ is unique on $\XbP$, the quasi-graph exists and is unique.

\begin{Prop}\label{quasi-graph inv}
The quasi-graph $U$ is $\skewT$-invariant.
\begin{proof}
By definition $U|_{\XbP}=\text{graph}(u(x))$ and we know that the graph of $u:\XbP\to\RR$ is $\skewT|_{\XbP}$-invariant, thus we only need to show that $\skewT U_x = U_{Tx}$ for all $x\in R_P$. Let $x\in R_P$. By Remark \ref{g orient pres}, we can assume each $g_x$ is orientation preserving. As $g_x$ is continuous and $\XbP$ is $T$-invariant,
\begin{align}\label{quasi inv eq}
\skewT\left(x,\limsup_{y \in \XbP, y\to x} u(y)\right) &= \left(Tx, \limsup_{y\in\XbP, y \to x}g_{y}(u(y))\right).
\end{align}
For all $y\in \XbP$, we have $g_y(u(y)) = u(Ty)$, so
\[
\skewT\left(x,\limsup_{y\in\XbP, y \to x} u(y)\right) = \left(Tx,\limsup_{y\in\XbP, y \to x}u(Ty)\right).
\]
As $\XbP$ is $T$-invariant and $T$ is continuous, as $y\to x$ we have $Ty\to Tx$ through a sequence in $\XbP$. So, $\limsup_{y\in\XbP, y \to x}u(Ty) = \limsup_{Ty \in \XbP, Ty\to Tx}u(Ty)$. By changing notation $Ty\mapsto y$, we have
\begin{align*}
\skewT\left(x,\limsup_{y\in\XbP, y \to x} u(y)\right) = \left(Tx, \limsup_{y \in \XbP, y\to Tx}u(y)\right).
\end{align*}
An identical argument holds for the $\liminf_{y\in\XbP, y \to x} u(y)$. By the Intermediate Value Theorem, for any $(x,t)\in U_x$ we have $\skewT(x,t)\in U_{Tx}$. Therefore, $\skewT (U_x) = U_{Tx}$.
Taking the union of all of these, as $R_P$ is $T$-invariant, $\skewT(U_{R_P}) = U_{R_P}$. We already know $\skewT(U_{\XbP}) = U_{\XbP}$, hence we have $\skewT(U) = U$.
\end{proof}
\end{Prop}

We now prove that if the quasi-graph is discontinuous at any point $p\in P$, then it is discontinuous on the dense set of pre-images of the fixed point $p$.

\begin{Prop}\label{lengths of Ux decrease}
Let $R_{p}$ denote the pre-images of a point $p\in P$. The length $|U_x|=0$ for some $x\in R_{p}$ if and only if $|U_x|=0$ for every $x\in R_{p}$.
\begin{proof}
Let $x\in R_{p}$ and $|U_x|=0$. We can define $u(x)= \lim_{y\in\XbP, y \to x}u(y)$. Thus, we have extended $u$ to the point $x\in R_{p}$ and $u$ is continuous at $x$. As $g_x$ is a continuous function and $g_x(u(x)) = u(Tx)$, so $u$ is continuous at $Tx$. Iterating, as $T^Nx = p$ for some $N\in\NN$, we have that $u$ is continuous at $p$.

Now, let $z\in R_{p}$. There exists $N$ such that $T^Nz=p$. By the invariance of $u$, $h^N_z(u(p))=u(z)$. As, $h^N_z$ is continuous, $u$ is continuous at $z$. So $|U_z|=0$. The proof of the other direction is trivial.
\end{proof}
\end{Prop}

\section{Proof of Theorems \ref{dichotomy thm} and \ref{open dense f}}

\subsection{Proof of Theorem \ref{dichotomy thm}}
\begin{proof}[Proof of Theorem \ref{dichotomy thm}]
Suppose that there exists $x\in R_P$ such that $|U_{x}|=0$.  Then, by
Proposition \ref{lengths of Ux decrease}, $|U_{p}|=0$ for some $p\in P$. We show that we can extend $u:\XbP\to\RR$ to an
$\al$-H\"{o}lder function $u_s:X\to\RR$, for some choice of $s \in \RR^\rho$, by iterating backwards towards
the fixed point $p$. As $U$ is the graph of $u_s$ the dichotomy is
proved.

Suppose $p\in X_j$. As $p$ is fixed under $T$ and $T$ is full
branched, there exists an inverse branch $\w:X\to X_j$ such that
$\w(p)=p$. As $\w$ is contracting, for any $x\in X$, $\w^nx\to p$ as
$n\to\infty$. Hence, $u(x) = g^n_{\w^nx}(u(\w^nx))$.  Let $x,y\in
\XbP$. Let $m$ be such that $\theta^m< \ep$. As $d(\w x, \w y) \leq \theta d(x,y)$, we have that
$\w^{m+i}x\in B_\ep(P)$ for all $i\in\NN$. Furthermore, by Lemma
\ref{aux lem bad set},
\begin{equation}\label{iterate far back enough close}
d(\w^{m+i}x,P)\leq \theta^i\ep.
\end{equation}
Using \eqref{non expanding}, \eqref{slow contracting} and \eqref{g non0},
\begin{align*}
|t-t'| &= |h^{n}_{\w^nx}(g^n_{\w^nx}(t))-h^{n}_{\w^nx}(g^n_{\w^nx}(t'))|\\
&= |h^{m}_{x}h^{n-m}_{\w^nx}(g^n_{\w^nx}(t))-h^{m}_{x}h^{n-m}_{\w^nx}(g^n_{\w^nx}(t'))|\\
&\geq C_{\mathcal{S}}^{-1}\la_{\min}^{m}|h^{n-m}_{\w^nx}(g^n_{\w^nx}(t))-h^{n-m}_{\w^nx}(g^n_{\w^nx}(t'))|\\
&\geq C_{\mathcal{S}}^{-1}\la_{\min}^{m}\la(\ep)\dots\la(\theta^{n-m}\ep)|g^n_{\w^nx}(t)-g^n_{\w^nx}(t')|.
\end{align*}
As $\log\la$ is $C^\al$ as a function of $\ep$ and $\log\la(0)=0$, using \eqref{iterate far back enough close},
\[
\left|-\sum_{i=0}^{n-m}\log\la(\theta^i\ep)\right|\leq\sum_{i=0}^{n-m}|\log\la(0)-\log\la(\theta^i\ep)|\leq \sum_{i=0}^\infty C_{\log\la}|\theta^{\al i}\ep| =: \log\Lambda.
\]
Exponentiating, $(\la(\ep)\dots\la(\theta^{n-m}\ep))^{-1} \leq \Lambda$ for some $\Lambda>0$ independent of $n$. As $m$ is fixed by $\ep$ and $T$, the $\la_{\min}^{m}$ term is absorbed into the constant. So, there exists $C_0>0$ independent of $x$, $n$ and $t$ such that
\begin{equation}\label{bounded when it back}
|g^n_{\w^nx}(t)-g^n_{\w^nx}(t')|\leq C_0\Lambda^{-1}|t-t'|.
\end{equation}
Consider
\begin{align}\label{thing to bound for C1 dichotomy}
\nonumber |u(x)-u(y)| &= |g^n_{\w^nx}(u(\w^nx))-g^n_{\w^ny}(u(\w^ny))|\\
\nonumber&\leq \sum_{i=0}^{n-1}|g^i_{\w^ix}g_{w_r^{i+1}x}g^{n-1-i}_{\w^ny}(u(\w^ny))-g^i_{\w^ix}g_{w_r^{i+1}y}g^{n-1-i}_{\w^ny}(u(\w^ny))|\\
&\hspace{2cm} + |g^n_{\w^nx}(u(\w^nx))-g^n_{\w^nx}(u(\w^ny))|.
\end{align}
Using \eqref{bounded when it back}, we know that
\begin{align*}
|g^n_{\w^nx}(u(\w^nx))-g^n_{\w^nx}(u(\w^ny))| \leq C_0\Lambda^{-1}|u(\w^nx)-u(\w^ny)|.
\end{align*}
As $\w^nx,\w^ny\to p$ and we assume that $u$ is continuous at $p$, this term tends to $0$ as $n\to\infty$. Letting $s=g_{w_r^{i+1}x}(t)$ and $s'=g_{w_r^{i+1}y}(t)$ where $t=g^{n-1-i}_{\w^ny}(u(\w^ny))$, we have
\begin{align*}
|g^i_{\w^ix}(s)-g^i_{\w^ix}(s')|\leq C_0\Lambda^{-1}|s-s'|.
\end{align*}
and
\begin{align*}
|s-s'|=|g_{w_r^{i+1}x}(t)-g_{w_r^{i+1}y}(t)|\leq C_g(t)d(w_r^{i+1}x,w_r^{i+1}y)^{\al}\leq C_g(t)\theta^{\al(i+1)}d(x,y)^{\al}.
\end{align*}
As $u$ is uniformly bounded and $\skewT$-invariant, we have $t=
u(\w^{n-1-i}y)$ and so $|t|\leq \|u\|_\infty$. Hence $t$ is contained
in a compact set $W\subset \RR$.
Thus
$C_g(t)\leq C_W$. As $\theta<1$, we can find a constant $C_u>0$
such that \eqref{thing to bound for C1 dichotomy} is bounded by
\[
|u(x)-u(y)|<C_ud(x,y)^{\al}.
\]
Therefore, $u$ is a uniformly $\al$-H\"{o}lder continuous function on
$\XbP$. As $\XbP$ is dense in $X$, we can uniquely extend $u$ to a
uniformly $\al$-H\"{o}lder continuous function $u_s : X \to \RR$ for
some $s\in\RR^\rho$. As $g,T$ and $u_s$ are continuous and $u$ is an
invariant graph on $\XbP$, $u_s$ must be an invariant graph on
$X$. Therefore, the quasi-graph $U=\textrm{graph}(u)$ is an
$\al$-H\"{o}lder continuous invariant graph of the skew product.
\end{proof}

\begin{Rem}\label{gn converges}
By \eqref{bounded when it back}, we have
\begin{align*}
|g^{n+1}_{\w^{n+1}x}(t)-g^n_{\w^nx}(t)|&=|g^n_{\w^{n}x}(g_{\w^{n+1}x}(t))-g^n_{\w^{n}x}(t)|\leq C_0\Lambda^{-1}|g_{\w^{n+1}x}(t)-t|.
\end{align*}
Furthermore, for $n$ sufficiently large, by \eqref{iterate far back enough close}, $d(\w^{n+1}x,p)\leq C\theta^{n-m}\ep$ for fixed $m$ independent of $x$. So,
\[
|g_{\w^{n+1}x}(t)-t| = |g_{\w^{n+1}x}(t)-g_p(t)|\leq C_g(t)\la(\theta^n\ep)\leq C_g(t)C\theta^{\al (n-m)}\ep^{\al}.
\]
As $\theta<1$, the sequence $g^n_{\w^nx}(t)$ is Cauchy (in $n$) and converges pointwise with $g^n_{\w^nx}(t)\to D_x(t)$ for some function $D_x:\RR\to\RR$. By \eqref{bounded when it back} we have $D_x$ is Lipschitz continuous as a function of $t$.
\end{Rem}

\subsection{Proof of Theorem \ref{open dense f}}
We prove this result in two parts, first the density of such skew products and second the openness.

\begin{Lem}\label{dense perturb f}
Let $\skewT\in \mathcal{S}(X,\RR, C^\al_P)$. There exists a $C^0$-dense set of $g\in C^\al_P$ such that $U$ is not the graph of a function.
\begin{proof}
By Theorem \ref{dichotomy thm}, it suffices to assume that there exists $s\in\RR^\rho$ such that the invariant graph $u_s$ of $\skewT\in \mathcal{S}(X,\RR,C^\al_P)$ is continuous at $p\in P$ and find a small perturbation of $g\mapsto \hat g$ within $C^\al_P$ such that, for all $s\in\RR^\rho$, the invariant graph $\hat u_s$ of $\tilde T_{\hat g}$ is not continuous at $p$. Equivalently, the quasi-graph of $\hat u$ is not the graph of a function. To do this, we assume that both $u_s$ and $\hat u_s$ are continuous for some $s,\hat s \in \RR^\rho$ and obtain a contradiction. As $u_s, \hat u_s$ are the unique, continuous extensions of $u, \hat u:\XbP\to\RR$, we denote the invariant graphs $u_s=u$, $\hat u_s=\hat u$

Let $q\not\in R_P$ be a periodic point. By Proposition \ref{fixed pts ok}, we can assume that $T$ is full branched with at least two inverse branches. Therefore, there exist disjoint sequences of pre-images of $q$, say $\w^nq\to p$ and $\nu^nq\to p$. As the orbit of $q\notin P\subset R_P$, for $n>N$ sufficiently large $\w^nq$ is sufficiently close to $p$ that $\w^nq$ cannot be in the orbit of $q$. Let $V$ be a small neighbourhood of $\w^j q$ for some $j> N$, disjoint from: $\nu^nq$ for all $n\in\NN$, $\w^nq$ for $n\not=j$, the orbit of $q$, and the set $P$. By making a $C^0$ perturbation $g\mapsto \hat g$ on $V$ with $\hat g\in C^\al_P$, we can ensure that
\[
|\hat g^j_q(u(p)) - g^j_q(u(p))|\geq \ep_1> 0
\]
for some $\ep_1>0$. We have not perturbed $g$ at $\w^nq$ for $n\not=j$, so $g_{\w^nq}=\hat g_{\w^nq}$. By inverting condition \eqref{non expanding}, for all $n>j$
\begin{align*}
|\hat g^n_q(u(p)) - g^n_q(u(p))|=|g^{n-j}_{T^jq}\hat g^j_q(u(p))- g^{n-j}_{T^jq}g^j_q(u(p))|\geq C_{\mathcal{S}}^{-1}\ep_1.
\end{align*}
By Remark \ref{gn converges}, as $n\to\infty$, we have
\begin{equation}\label{Ds differ}
|D_q(u(p))-\hat D_q(u(p))|\geq C_{\mathcal{S}}^{-1}\ep_1.
\end{equation}
As $V$ is disjoint from the orbit of $q$, we have $h_{T^iq}=\hat h_{T^iq}$ for all $i\geq 0$. As $q\notin R_P$, by Theorem \ref{existence and uniqueness of la1}, $u(q)=\hat u(q)$.

As $V$ is, also, disjoint from $\nu^iq$ for all $i\geq 0$, we also have $g_{\nu^iq}=\hat g_{\nu^iq}$ for all $i\geq 0$. So, $g^n_{\nu^n q}(u(\nu^nq))=u(q)$
and $\hat g^n_{\nu^n q}(\hat u(\nu^nq))=g^n_{\nu^n q}(\hat u(\nu^nq)) = \hat u(q)$. Thus, $u(\nu^nq)=\hat u(\nu^nq)$ for all $n\in\NN$. In particular, as $u,\hat u$ are continuous,
\begin{equation}\label{up equal up}
\hat u(p) = u(p).
\end{equation}
As both $g^n_{\w^n q}(u(\w^nq)) = u(q)$ and $\hat g^n_{\w^n q}(\hat u(\w^nq)) = \hat u(q)$, we obtain
\begin{align}\label{to show stuff not equal}
\nonumber 0= u(q)-\hat u(q) &= g^n_{\w^nq}(u(\w^nq))-\hat g^n_{\w^nq}(\hat u(\w^nq))\\
&=g^n_{\w^nq}(u(\w^nq))-\hat g^n_{\w^nq} (u(\w^nq))+\hat g^n_{\w^nq} (u(\w^nq))-\hat g^n_{\w^nq}(\hat u(\w^nq)).
\end{align}
As $u$ and $\hat u$ are continuous, by assumption, and $D_q$ is Lipschitz continuous, we have $g^n_{\w^nq} (u(\w^nq))\to D_q(u(p))$ and similarly for $\hat D_q(u(p))$. By \eqref{Ds differ}, for all $n>N$ with $N$ sufficiently large, there exists $\de_0>0$ such that
\begin{equation}\label{stuff not equal}
|g^n_{\w^nq}(u(\w^nq))-\hat g^n_{\w^nq} (u(\w^nq))|>\de_0.
\end{equation}
Substituting \eqref{stuff not equal} into \eqref{to show stuff not equal},
\[
|\hat g^n_{\w^nq} (u(\w^nq))-\hat g^n_{\w^nq}(\hat u(\w^nq))|>\de_0,
\]
for all $n> N$. Therefore, $\hat D_q(u(p))\not=\hat D_q(\hat u(p))$, hence $u(p)\not= \hat u(p)$, contradicting \eqref{up equal up}. As the perturbation $g\mapsto \hat g$ is arbitrarily small in the uniform norm, we have a $C^0$-dense set of $g\in C_P^\al$ such that for all $s\in\RR$ the invariant graph $u_s$ is not continuous.
\end{proof}
\end{Lem}

We now show that if $U$ is not the graph of an invariant function $u$, then for any small $C^\al$ perturbation of $g\mapsto \hat g$ within $C^\al_P$, the quasi-graph $\hat U$ is, also, not the graph of an invariant function $\hat u$.

\begin{Lem}\label{open perturb f}
Let $\skewT\in \mathcal{S}(X,\RR, C^\al_P)$ be such that the invariant quasi-graph is not the graph of a function. There is a $C^\al$-open set of perturbations $g\mapsto \hat g$ such that the invariant quasi graph $\hat U$ of $\tilde T_{\hat g}$ is not the graph of a function.
\begin{proof}
Suppose that, for a given $g\in C^\al_P$, $U$ is not the graph of a function. Let $p\in P$. By Theorem \ref{dichotomy thm}, there exists $\ga>0$ such that $|U_{p}|\geq\ga$. Let $\ka\in \NN$ and let $\de_\ka>0$ be a sequence such that $\de_\ka\to 0$ as $\ka\to\infty$. Let $0<\ga_0<\ga$. Let $B_{\de_\ka}(p)$ be the open ball of radius $\de_\ka$ about $p$. As $U$ is not the graph of a continuous function, for any $\ka$ there exists $x_\ka, y_\ka\in B_{\de_\ka}(p)\setminus R_P$ such that $|u(x_\ka) - u(y_\ka)|\geq \ga_0$.

We prove that for a sufficiently small perturbation, there exists $\hat\ga>0$ such that $|\hat U_{p}|\geq\hat\ga>0$. We first claim that, for any $\ep_0>0$ there is a $C^\al$-open set of perturbations of $g\mapsto \hat g$ within $C^\al_P$, made precise in (\ref{define open perturb}), such that, for all $x\in \XbP$,
\begin{equation}\label{claim open}
|u(x) - \hat u(x)|<\ep_0,
\end{equation}
where $\hat u:\XbP\to\RR$ is the invariant graph of the perturbed skew product. The claim proves the lemma as, for all $\ka\in\NN$,
\begin{align*}
0<\ga_0\leq|u(x_\ka)-u(y_\ka)| &\leq |u(x_\ka)-\hat u(x_\ka)|+|u(y_\ka)-\hat u(y_\ka)|+|\hat u(x_\ka) - \hat u(y_\ka)|\\
		&\leq |\hat u(x_\ka) - \hat u(y_\ka)| + 2\ep_0,
\end{align*}
by the claim. As $\ep_0$ is arbitrarily small and $\ga_0$ can be chosen to be close to $\ga$, we can ensure that $0<\ga_0-2\ep_0 = \hat\ga$. Hence, there exists sequence $x_\ka, y_\ka$ such that $|\hat u(x_\ka) - \hat u(y_\ka)|\geq\hat\ga>0$. As $\ka\to\infty$, $x_\ka, y_\ka\to p$. Therefore, $\lim_{y \in\XbP, y\to p}\hat u(y)$ does not exist, so $|\hat U_{p}|>0$. Hence, $\hat U$ is not the graph of a function by Theorem \ref{dichotomy thm}.

It remains to prove the claim \eqref{claim open}. Let $\de>0$ be sufficiently small that Lemma \ref{aux lem bad set} holds. Again, let $G=X\setminus B_\de(P)$.

Using the notation of Theorem \ref{existence and uniqueness of la1}, let $j$ be the number of visits of the orbit segment $x,\dots, T^nx$ to $G$, denoting the $j^{th}$ visit by $T^{n_j}x$. As before, if $x\notin G$, let $n_{-1}=0$ so that $n_0$ is the first visit of the orbit of $x$ to $G$. Write $H_{T^{n_{j-1}}x} = h^{n_j-n_{j-1}}_{T^{n_{j-1}}x}$. By Theorem \ref{existence and uniqueness of la1} and passing to the subsequence $n_j$, for any $t\in\RR$,
\[
|\hat u(x)-u(x)|=\left|\lim_{j\to\infty}\hat H^{j}_x(t)-\lim_{j\to\infty} H^{j}_x(t)\right| = \lim_{j\to\infty}|\hat H^{j}_x(t)- H^{j}_x(t)|.
\]
As the limits are independent of $t$, we choose $|t|\leq \|u\|_\infty$. By \eqref{fast contracting},
\begin{align}\label{left to bound h}
\nonumber |\hat H^{j}_x(t)- H^{j}_x(t)|&\leq\sum_{i=0}^{j-1}|\hat H^{i}_x\hat H_{T^{n_i}x}H^{j-i}_{T^{n_{i+1}}x}(t)-\hat H^{i}_xH_{T^{n_i}x}H^{j-i}_{T^{n_{i+1}}x}(t)|\\
&\leq \sum_{i=0}^{j-1}C_{\mathcal{S}}\la_\de^i|\hat H_{T^{n_i}x}H^{j-i}_{T^{n_{i+1}}x}(t)-H_{T^{n_i}x}H^{j-i}_{T^{n_{i+1}}x}(t)|
\end{align}
For notation, fix $s_{i,j}=H^{j-i}_{T^{n_{i+1}}x}(t)$, $y=T^{n_i}x$ and $m=n_{i+1}-n_i$. Then bound
\begin{align}\label{bounding h first steps}
\nonumber |\hat H_y(s_{i,j})-H_y(s_{i,j})| &= |\hat h^{m}_y(s_{i,j})-h^{m}_y(s_{i,j})|\\
\nonumber &\leq \sum_{k=0}^{m-1}|\hat h^{k}_y\hat h_{T^k y}h^{m-1-k}_{T^{k+1}y}(s_{i,j})-\hat h^{k}_yh_{T^k y}h^{m-1-k}_{T^{k+1}y}(s_{i,j})|\\
&\leq C_{\mathcal{S}}\sum_{k=0}^{m-1}|\hat h_{T^k y} h^{m-1-k}_{T^{k+1}y}(s_{i,j})-h_{T^k y}h^{m-1-k}_{T^{k+1}y}(s_{i,j})|.
\end{align}
As $T^k y\in B_\ep(P)$ for all $k$ such that $0\leq k\leq m-1$, by
Lemma \ref{proving f bounded on bad orbits}, letting $r_k(s_{i,j}) =
h^{m-1-k}_{T^{k+1}y}(s_{i,j})$ we see that $|r_k-s_{i,j}|$ is bounded
independently of $k$ and $m$. To show that $s_{i,j}$ is bounded
independently of $i$ and $j$ we consider $z=T^{n_{i+1}}x$, $M=j-i$. By
Lemma \ref{proving f bounded on bad orbits} we have
\begin{align}\label{bound hm}
\nonumber|H^{M}_z(t)-t|&\leq \sum_{\ell=0}^{M-1}|H^{\ell}_zH_{T^{n_i}z}H^{M-\ell-1}_{T^{n_{\ell+1}}p}(t)-H^{\ell}_zH_{T^{n_{\ell}}p}H^{M-\ell-1}_{T^{n_{\ell+1}}p}(t)|\\
&\leq \sum_{\ell=0}^{M-1}\la_\de^\ell|H_{T^{n_i}z}(t)-t|\leq \sum_{\ell=0}^{M-1}\la_\de^\ell A_h(t)\leq C_0(t),
\end{align}
recalling that $H_{p}$ is the identity as $p$ is a fixed point. Therefore, we have $|s_{i,j}-t|=|H^{j-i}_{T^{n_{i+1}}x}(t)-t|\leq C_0(t)$ bounded independently of $i$, $j$ and $n$ and so $|s_{i,j}|\leq C_0(t)+|t|$. By choice of $t$, $|t|\leq \|u\|_\infty$. Hence we can find a compact set $W\subset \RR$ containing each $s_{i,j}$, $r_k(s_{i,j})$ and $t$ which depends only on the skew product.

For fixed $t\in\RR$, the maps $x\mapsto h_x(t)$, $x\mapsto\hat h_x(t)$ are $\al$-H\"{o}lder and both $g_{p}$ and $\hat g_{p}$ are the identity map. Hence,
\begin{align}\label{define open perturb}
|\hat h_x(t)- h_x(t)| &= |\hat h_x(t) -h_x(t) + \hat h_{p}(t) - h_{p}(t)| \leq C_{(h-\hat h)}(t)d(x,p)^{\al}.
\end{align}
For any $t\in W$ where $W$ is compact and fixed, if the perturbation is sufficiently small in the $C^{\al}$ topology, we can ensure that $C_{(h-\hat h)}(W) = \sup_{t\in W} \left(C_{(h-\hat h)}(t)\right)$ is small. As $r_k(s_{i,j})\in W$,
\begin{equation*}
|\hat h_{T^k y}(r_k)-h_{T^k y}(r_k)|\leq C_{(h-\hat h)}(W)d(T^k y, p)^{\al}.
\end{equation*}
As $T^k y\in B_\ep(P)$ for all $0\leq k \leq m-1$, by Lemma \ref{aux lem bad set}, $d(T^k y,p)\leq \ep\theta^{m-k}$. So we can bound
\begin{align}\label{sum bound g}
\sum_{k=0}^{m-1}|\hat h_{T^k y} (r_k)-h_{T^k y}(r_k)|&\leq \sum_{k=0}^{m-1}C_{(h-\hat h)}(W)\ep\theta^{\al(m-k)}\leq C_1C_{(h-\hat h)}(W).
\end{align}
for some $C_1>0$. Substituting \eqref{sum bound g} into \eqref{bounding h first steps} and then \eqref{left to bound h}, we have
\begin{align*}
|\hat u(x) - u(x)|&\leq \sum_{i=0}^\infty C_{\mathcal{S}}\la_\de^i C_2C_{(h-\hat h)}(W).
\end{align*}
for $C_2>0$. As $\la_\de<1$, for a sufficiently small $C^{\al}$ perturbation of $g$ within $C^\al_P$, we can ensure that $|\hat u(x)- u(x)|<\ep_0$, thus completing the proof of the claim.
\end{proof}
\end{Lem}

\begin{proof}[Proof of Theorem \ref{open dense f}.]
By combining Lemma \ref{open perturb f} and Lemma \ref{dense perturb f}, we have a $C^\al$-open and $C^{0}$-dense set of $g\in C^\al_P$ such that the invariant quasi-graph $U$ of $\skewT$ is not the graph of a function.
\end{proof}

\section{Proof of Theorem \ref{dichotomy thm strong}}\label{expanding endo}
A special case of a piecewise $C^1$ Markov map of the circle $X$ is an
orientation preserving $C^{1+\al}$ circle endomorphism for $\al>
0$. Notice that each inverse branch is full and choose the end-points
of each $X_j$ in the Markov partition to be given by $t_j$ such that
$T(t_j)=0$. Let $r\in\NN$, $r\geq 1$ and $\al>0$. Denote the set of
skew products with $T\in C^{r+\al}$ and $g\in C^{r+\al}_P$ by
$\skewT\in \mathcal{S}_{C^{r+\al}}(X,\RR,C^{r+\al}_P).$ Throughout
this section we will refer to the base $X$ as the `horizontal' and the
fibre $\RR$ as the `vertical' directions, respectively.  To prove
Theorem \ref{dichotomy thm strong}, we use stable manifold theory (cf.\
\cite{HPS, Fathi}).

\subsection{The local stable manifold}
Let $r\in\NN$, $r\geq 1$ and $\al> 0$. Let $\skewT\in \mathcal{S}_{C^{r+\al}}(X,\RR,C^{r+\al}_P)$. Let $(p,u(p))$ with $p\in P$ a fixed point of $\skewT$. Then, there exists an inverse branch $\tilde\w(x,s)=(\w x, h_{\w x}(s))$ also fixing $(p,u(p))$ with derivative
\begin{align*}
D\tilde\w(p,u(p)) &= \begin{pmatrix}
\w'(p) &0\\
\del (h_{\w p}(u(p)))\w'(p) &\del (h_{\w p}(u(p)))
\end{pmatrix}
=
\begin{pmatrix}
\w'(p)& 0\\
\w'(p)& 1
\end{pmatrix}.
\end{align*}
Let $E$ be the tangent space of $X\times\RR$ at $(p,u(p))$ and let $E_1, E_2$ be the eigenspaces corresponding to the eigenvalues $\w'(p)$ and $1$ respectively. We can think of $E_2$ as consisting of vertical vectors.

Let $\ga_1, \ga_2$ be the linear maps in the direction of the eigenvalues $\w'(p)$ and $1$ of the derivative of $\tilde\w$ at the fixed point $(p,u(p))$. Let $\ga:E\to E$ be such that $\ga|_{E_i}=\ga_i$. Clearly $\ga_i(E_i)=E_i$. Let $\theta<\chi<1$ where $\theta^{-1}=\inf_{x\in X}|T'(x)|$. We state the following theorem, which is true in a much wider setting.

\begin{Thm*}\cite[Theorem A.6]{Fathi}
Let $\xi:E\to E$ be a $C^{r+\al}$ map with fixed point at the origin such that the Lipschitz constant $\textnormal{Lip}(\ga - \xi)<\min\{\chi-\|\ga_1\|,\|\ga_2^{-1}\|^{-1}-\chi\}$.
The set
\[
W^{s,\chi}_\xi = \left\{x\in E \mid \sup_{n\geq 0}\|\chi^{-n}\xi^n(x)\|<\infty\right\}
\]
is the graph of a $C^{r+\al}$ function $\zeta:E_1\to E_2$. Moreover, if $x\in W_\xi^{s,\chi}$, then $\chi^{-n}\xi^n(x)\to 0$ as $n\to\infty$. Furthermore, if the differential $D\xi(0)=\ga$, then $W^{s,\chi}_\xi$ is tangent to $E_1$.
\end{Thm*}
Let $\ep>0$ be small and let $N_0\subset E$ be a neighbourhood of the origin with diameter at most $\ep$. Let $\text{Exp}:N_0\to \RR^2$ be the smooth exponential map from a subset of $E$ to $X\times\RR$ with $\text{Exp}(0,0) = (p,u(p))$. Let $\xi=\text{Exp}^{-1}(\tilde\w)$. Let $\ep$ be sufficiently small that the restriction of $\xi-\ga$ to $N_0$ satisfies the conditions of \cite[Theorem A.6]{Fathi}.

Applying the exponential map to $W^{s,\chi}_\xi$ we obtain a $C^{r+\al}$ immersed manifold $\tilde W^{s,\chi}_\ep\subset X \times\RR$ contained in a ball of radius $\ep$ centred at $(p,u(p))$, say $B_\ep(p,u(p))$, such that for $(x,s)\in \tilde W^{s,\chi}_\ep$ we have
\begin{equation}\label{elements of tildeW}
\chi^{-n}d(\tilde\w^n(x,s),(p,u(p)) \to 0,
\end{equation}
as $n\to\infty$. We call $\tilde W_{\ep}^{s,\chi}$ the local stable manifold through the fixed point.

\subsection{Proof of Theorem \ref{dichotomy thm strong}}
We assume that $U$ is the graph of a continuous function. By Theorem \ref{dichotomy thm}, $U$ is Lipschitz continuous. We extend the stable manifold such that it is both $C^{r+\al}$ and equal to the Lipschitz continuous quasi-graph $U$.

We show that $\skewT(\tilde W^{s,\chi}_\ep)\supset\tilde W^{s,\chi}_\ep$. Consider, $\skewT(\tilde W^{s,\chi}_\ep)=\{(y,t) = \skewT(x,s)\mid (x,s)\in \tilde W^{s,\chi}_\ep\}$. As $\skewT$ is $C^{r+\al}$, this is a $C^{r+\al}$ immersed manifold. Suppose that $(y,t)\in\tilde W^{s,\chi}_\ep$. Let $(x,s)=\tilde\w(y,t)$. Let $B_\ep(p,u(p))$ be a small ball about the fixed point such that $\tilde W^{s,\chi}_\ep\subset B_\ep(p,u(p))$. As $(y,t)\in B_\ep(p,u(p))$ and $\tilde\w$ contracts, we have that $(x,s)\in B_\ep(p,u(p))$. To see that $(x,s)$ is in the local stable manifold, consider the lift to the tangent space $\text{Exp}^{-1}(x,s) = v_{(x,s)}$ and $\text{Exp}^{-1}(y,t) = v_{(y,t)}$. Notice that
\[
\|\chi^{-n}\xi^n(v_{(x,s)})\|=\|\chi^{-n}\xi^{n+1}(v_{(y,t)})\|=\chi\|\chi^{-(n+1)}\xi^{n+1}(v_{(y,t)})\|\to 0,
\]
as $v_{(y,t)}\in W^{s,\chi}_\ep$. Therefore, $v_{(x,s)}\in W^{s,\chi}_\ep$. Applying the exponential map, we see that $(x,s)\in \tilde W^{s,\chi}_\ep$. Therefore, $(y,t)\in \skewT(\tilde W^{s,\chi}_\ep)$, so $\tilde W^{s,\chi}_\ep\subset\skewT(\tilde W^{s,\chi}_\ep)$.

Let $N$ be the least integer for which there exists a cylinder $C_N\subset [0,1]$ of rank $N$ with $C_N\subset B_\ep(p)$. Then, $T^N(C_N) = [0,1]$. Define $\tilde W^{s,\chi}_{C_N}$, the restriction of $\tilde W^{s,\chi}_\ep$ to $C_N$. Furthermore, as $E_1$ is transverse to $E_2$, the immersed manifold $\tilde W^{s,\chi}_{C_N}$ is the graph of a function $X\to \RR$. By applying the skew product iteratively we can define a surjective $C^{r+\al}$ function $[0,1]\to \RR$ by $\tilde W=\skewT^N(\tilde W^{s,\chi}_{C_N})$, where $\tilde\w(\tilde W)\subset \tilde W$. We, finally, prove the strong dichotomy.

\begin{proof}[Proof of Theorem \ref{dichotomy thm strong}.]
It suffices to show that $U\subset \tilde W$ are both are continuous
graphs over $[0,1]$, then check continuity at $1\cong 0$. Suppose
$(x,u(x))\in U$ but $(x,u(x))\not\in \tilde W$. If $(x,u(x))\in U$
then $\tilde\w(x,u(x)) = (\w x, u(\w x))\in U$, similarly
$\tilde\w(\tilde W)\subset \tilde W$. By applying the pre-image
$\tilde\w$ repeatedly, it suffices to assume $x\in C_N$ with
$(x,u(x))\not\in\tilde W^{s,\chi}_{C_N}$. By definition, if
$(x,u(x))\not\in \tilde W^{s,\chi}_{C_N}$, there exists subsequence
$n_m$ such that
$\chi^{-{n_m}}d(\tilde\w^{n_m}(x,u(x)),(p,u(p)))\to\infty$. We know
that $\theta^{-n_m}>\chi^{-n_m}$ and
$d(\w^{n_m}x,p)^{-1}\geq\theta^{-n_m}$. Therefore,
\begin{equation}\label{cr lipsch contra}
\frac{d((\w^{n_m}x,u(\w^{n_m}x)),(p,u(p)))}{d(\w^{n_m}x, p)}\to\infty,
\end{equation}
contradicting that $u$ is Lipschitz continuous. Therefore, $U=\tilde W$; so, $U$ is the graph of a $C^{r+\al}$ function on $[0,1]$. To show that $\tilde W$ is $C^{1+\al}$ at $0$, let $V$ be an open ball about a pre-image of $0$ not containing $0$; as $T$ has $b\geq 2$ branches, such a set exists. We know that $\tilde W$ is $C^{r+\al}$ on $V$. As $\skewT$ is $C^{r+\al}$ and $\tilde W$ is $\skewT$ invariant, it follows that $U=\tilde W$ is $C^{r+\al}$ at $0$ as required.
\end{proof}

\section{The box dimension of quasi-graphs}
In this section, we calculate the box dimension of the quasi-graphs of partially hyperbolic skew products $\skewT\in \mathcal{S}(X,\RR, C^{\Lip,2}_P)$. Throughout, we will use the notation $U_{C_n}$ to denote the restriction of the skew product to the cylinder $C_n$, thus $|U_{C_n}|$ is the height of the quasi-graph over this cylinder and $C_n\times U_{C_n} \subset X\times\RR$ is the smallest rectangle containing every element of $U$ over the cylinder $C_n$. By Proposition \ref{fixed pts ok} and Remark \ref{g orient pres}, we can assume that $P$ consists of fixed points, $T$ is full branched and $\del g_x (t)>0.$ The main difficulty is in establishing various bounded distortion estimates; once we obtain those, the argument follows that in \cite[Theorem 13.1]{Pesin}.

As we only assume $T$ is a Markov map, we abuse notation slightly by defining $\|T'\|_\infty$ and $\m(|T'|)$ as the supremum and infimum of $|T'(x)|$ for $x\in\bigcup_{i=0}^{b-1}\text{Int}(X_j)$ respectively. Note that $\m|T'| = \theta^{-1}$. Also, we notice that conditions \eqref{non expanding}, \eqref{slow contracting} and \eqref{fast contracting} can be written as bounds on the derivative of $h^n_x(t)$, see Remark \ref{examples and constant}. In particular, if the orbit segment $x,\dots, T^{n-1}x$ visits $X\backslash B_\de(P)$ $j$-many times, there exists $0<\la_\de<1$ such that, for all $(x,s)\in X\times\RR$,
\begin{equation}\label{fast contr diff}
|\del(h^n_x(s))|\leq C_\mathcal{S}\la_\de^j.
\end{equation}

\subsection{Bounds on the height of cylinders}\label{bounds on height}
To calculate the box dimension of graphs it is important to estimate the height of the graph over a cylinder. We omit details of the argument for the lower bound as the argument follows the uniformly expanding version which can be found as \cite[Proposition 3.1]{MossWalkdenDim} and \cite[Proposition 8]{Bedford}.

\begin{Lem}\label{upper sup bound on height}
Let $\skewT\in \mathcal{S}(X,\RR, C^{\Lip,2}_P)$ be partially hyperbolic. There exists constant $A>0$ independent of $C_n$ such that
\[
\Height(U_{C_n})\leq A\sup_{(x,s)\in C_n\times U_{C_n}}\prod_{j=0}^{n-1}\del h_{T^jx}(g^{j+1}_x(s)).
\]
\end{Lem}
\begin{proof}
See \cite[Proposition 3.1]{MossWalkdenDim} for the affine case or \cite[Proposition 8]{Bedford} for $g_x$ a diffeomorphism, and Remark \ref{non vertical obs}.
\end{proof}

\begin{Rem}\label{non vertical obs}
In \cite[Proposition 3.1]{MossWalkdenDim}, a straight line path
$\hat\ga$ joining the supremum and infimum of the height of the
invariant graph $u$ over the cylinder is defined and it is proved that
both the horizontal and vertical path integrals along
$\skewT^n\hat\ga=\ga$ are bounded. In our setting, $\hat\ga$ could be
a vertical line; when the infimum and supremum of the quasi-graph over
$C_n$ happen to both be at the same point of $R_P$.  In this case, as
$U$ is invariant, $\skewT^n\hat\ga=\ga$ is also a vertical line, so
the horizontal derivative is $0$ and the vertical derivative is
bounded by $\int|\ga'_V(a)|\;da\leq 2\|u\|_\infty$, where $\ga_V$ is
the projection of $\ga$ onto the vertical axis. This observation
allows us to continue to use Lemma \ref{upper sup bound on height} in
our setting.
\end{Rem}

The lower bound is actually more straightforward, employing the properties of quasi-graphs.

\begin{Lem}\label{lower bound on height}
Let $\skewT\in \mathcal{S}(X,\RR, C^{\Lip,2}_P)$ be partially hyperbolic. There exists constant $B>1$ independent of $C_n$ such that
\[
\Height(U_{C_n})\geq B^{-1}\inf_{(x,s)\in C_n\times U_{C_n}}\prod_{j=0}^{n-1}\del h_{T^jx}(g^{j+1}_x(s)).
\]
\begin{proof}
As $U$ is not a continuous function of $X$, we have $|U_x|>0$ for
every $x\in R_P$, by Theorem \ref{dichotomy thm}. Let $x\in C_n$ be a
pre-image of $p\in P$ such that $T^nx=p$. Therefore $|U_x|>0$ and
$g^n_x(U_x)=U_p$. Denote the end-points of $|U_x|$ by $t, s$. As $U$
is invariant, the end-points of $U_x$ map to the end-points of
$U_p$. Let $t'=\sup (U_p)$, then we have $h^{n}_x(t') = t$;
similarly let $s'=\inf (U_p)$, so that $h^{n}_x(s') = s$. Hence, as $h^{n}_x$
is a diffeomorphism, the Mean Value Theorem gives:
\begin{align*}
\Height(U_{C_n})\geq |t-s|=|h^{n}_x(t')-h_x^{n}(s')|\geq \inf_{r'\in U_p}|\del (h^{n}_x(r'))||U_p|.
\end{align*}
As $U_p$ is a closed interval the infimum is attained, say at the point $(x,r')$.
As the fibres map bijectively to each other, $r' = g^n_x(r)$ for some $r\in U_x$, so
\begin{align*}
\del (h^{n}_x(r')) = \del(h_{x}\dots h_{T^{n-1}x}(r')) = \prod_{j=0}^{n-1}\del h_{T^jx}(h_{T^{j+1}x}\dots h_{T^{n-1}x}(g^n_x(r))) = \prod_{j=0}^{n-1}\del h_{T^jx}(g^{j+1}_x(r)),
\end{align*}
as required.
\end{proof}
\end{Lem}
Combining the two lemmas we have the following result.

\begin{Prop}\label{bound on height}
Let $\skewT\in S(X,\RR, \bar{C}^2_P)$ be partially hyperbolic. Let $C_n$ be a cylinder of rank $n$. There exists a constant $D_1>1$, independent of $C_n$, such that,
\begin{align*}
D_1^{-1}\inf_{(x,s)\in C_n\times U_{C_n}}\prod_{j=0}^{n-1}\del h_{T^jx}(g^{j+1}_x(s))\leq \Height(U_{C_n})\leq D_1\sup_{(x,s)\in C_n\times U_{C_n}}\prod_{j=0}^{n-1}\del h_{T^jx}(g^{j+1}_x(s)).
\end{align*}
\end{Prop}

We need to be able to strengthen this result to show that there is no need to take the infimum and supremum, rather there exists a constant independent of $(x,t)$ such that the bound holds for any $(x,s)\in C_n\times U_{C_n}$. We do this by proving bounded distortion results in the vertical direction.

\subsection{Continuity of the quasi-graph on $\XbP$}
We can use the upper bound on the heights of the quasi-graph over cylinders to prove continuity of $U$ at points in $\XbP$.  The following is a standard fact.

\begin{Prop}\label{bounded distortion}\cite[Lemma 5]{Bedford}, \cite[Proposition 4.2]{Falconer2}
Let $\al>0$. Let $T$ be a $C^{1+\al}$ expanding circle endomorphism. $C_n\subset X$ be  a cylinder of rank $n$. There exists $D_0>0$ such that for any $x\in C_n$,
$D_0^{-1}\left(|T'|^n(x)\right)^{-1}\leq \diam(C_n) \leq D_0\left(|T'|^n(x)\right)^{-1}$.
\end{Prop}

We now prove that the invariant graph of a partially hyperbolic skew product is continuous on $\XbP$.
\begin{Prop}\label{u connected}
Let $\skewT\in \mathcal{S}(X,\RR, C^{\Lip,2}_P)$. The quasi-graph $U$ is continuous at $x\in\XbP$.
\begin{proof}
Let $x\in \XbP$. Suppose $M\in\NN$. Let $\ep>0$ be as in Lemma \ref{aux lem bad set} and let $\de>0$ be such that $\de<\ep/2$. Let $n$ be the number of visits of the orbit segment $x,\dots, T^{M-1}x$ to $X\setminus B_\ep(P)$. By Corollary \ref{cor to aux lem}, as $M\to\infty$, we have $n\to\infty$.

Fix $\eta > 0$ small. Let $M$ be sufficiently large that $D_1C_{\mathcal{S}}\la_\de^n<\eta$ where $D_1$ is as in Lemma \ref{bound on height} and $\la_\de$, $C_\mathcal{S}$ are as in condition \eqref{fast contr diff}. Let $D_0$ be the constant from bounded distortion, Proposition \ref{bounded distortion}. Let $C_N$ be a cylinder of rank $N\geq M$ containing $x$ such that $\diam(C_N) \leq D_0^{-1}\|T'\|_\infty^{-M}\de$. Therefore, for $0\leq j\leq M-1$, if $T^jx\in X\setminus B_\ep(P)$, then the intersections $T^j(C_N)\cap B_\de(P)=\emptyset$. By Proposition \ref{bound on height} and condition \eqref{fast contr diff},
\begin{align*}
\Height(C_N)&\leq D_1\sup_{(y,s)\in C_N\times U_{C_N}}\prod_{j=0}^{N-1}\del h_{T^jy}(g^{j+1}_y(s))\leq D_1C_{\mathcal{S}}\la_\de^n <\eta.
\end{align*}
Thus, $x$ is continuous on $\XbP$.
\end{proof}
\end{Prop}

Using continuity, we can prove that every point in the vertical cylinder $U_{C_n}$ is attained at some point by the quasi-graph.

\begin{Prop}\label{every vertical point attained}
Let $t\in U_{C_n}$.  Then there exists $x\in C_n$ such that $t\in U_x$.
\begin{proof}
Let  $s=\sup (U_{C_n})$ and let $(x,s)\in U$. As $x\in C_{j,n}$ for some $j$, we must have $s=\sup (U_{C_{j,n}})$. A similar statement holds for the lower boundary. Let $t\in U_{C_n}$. Suppose $t\notin U_{C_{j,n}}$ for all $j$ such that $0\leq j\leq b-1$, then we must have two adjacent cylinders of the same rank $C_{j_0,n}, C_{j_1, n}\subset C_n$ such that $\sup (U_{C_{j_0,n}}) <t<\inf (U_{C_{j_1,n}})$.

The cylinders $C_{j_0,n}$, $C_{j_1,n}$ intersect at a single point, say $z$.  Hence
\[
\liminf_{y \in \XbP, y\to z}u(y)<t<\limsup_{y\in \XbP, y\to z}u(y).
\]
If $z\in \XbP$, this contradicts the continuity of $u(z)$ on $\XbP$ guaranteed by Proposition \ref{u connected}; otherwise $z\in R_P$, so $t\in U_z$. Hence, we can find $j$ such that $t\in U_{C_{j,n}}$. By iteration, for any $N>n$, there exist nested cylinders $C_N$ of rank $N$ such that $t\in U_{C_N}\subset U_{C_n}$. Letting $N\to\infty$, we have $C_N\searrow \{y\}$ for some $y\in C_n$.  Hence $t\in U_y$.
\end{proof}
\end{Prop}

\subsection{Bounded distortion}\label{bdd distortion section}
Our first task is to prove that a point in a vertical cylinder iterated under the dynamics is not repelled too far away from the quasi-graph

\begin{Lem}\label{itrt forw t bdd}
Let $\skewT$ be partially hyperbolic. Suppose $t\in U_{C_n}$. There exist intervals $V\subset \RR$ depending only on the skew product such that, for any $x\in C_n$, we have $g^n_x(t)\in V$.
\begin{proof}
As $t\in U_{C_n}$, by Proposition \ref{every vertical point attained} there exists $y\in C_n$ such that $t\in U_y$. Therefore, by Corollary \ref{existence of the us}, there exists $s\in\RR^\rho$ such that $u_s:X\to \RR$ with $u_s=u$ on $\XbP$, such that $u_s(y)=t$. We need to extend the invariant graph as $y$ may be in $R_P$. By the Mean Value Theorem,
\begin{align*}
|g^n_x(t)-g^n_y(t)|&=|g^n_x(u_s(y))-g^n_y(u_s(y))|\\
&\leq \sum_{j=0}^{n-1}\left|g_{T^{n-1}x}\dots g_{T^{j+1}x}g_{T^jx}(u_s(T^{j}y))-g_{T^{n-1}x}\dots g_{T^{j+1}x}g_{T^jy}(u_s(T^{j}y))\right|\\
&\leq \sum_{j=0}^{n-1}\|\del g_{T^{n-1}x}\|_\infty\dots\|\del g_{T^{j+1}x}\|_\infty|g_{T^jx}(u_s(T^{j}y))-g_{T^jy}(u_s(T^{j}y))|\\
&\leq \sum_{j=0}^{n-1}\|\del g_{T^{n-1}x}\|_\infty\dots\|\del g_{T^{j+1}x}\|_\infty C_Wd(T^jx,T^jy),
\end{align*}
where $C_W$ is the supremum of the Lipschitz constants of the functions $x\mapsto g_x(r)$ for $r\in[-\|u_s\|_\infty, \|u_s\|_\infty]$, $u_s$ is bounded by $u$ and the choice of $t\in U$. As $d(T^jx, T^jy)\in C_{n-j}$ for some cylinder of rank $n-j$, using Proposition \ref{bounded distortion}, we have
\[
|g^n_x(t)-g^n_y(t)| \leq \sum_{j=0}^{n-1}C_W D_0\|\del g\|_\infty^{n-j}\left(\m|T'|^{n-j}\right)^{-1}
\]
where $\|\del g\|_\infty \leq\ka^{-1} \m|T'|$ by the partial
hyperbolicity assumption \eqref{partial hyp}.  Hence the sum converges and
\[
|g^n_x(t)-g^n_y(t)|\leq \frac{C_WD_0}{1-\ka^{-1}}=C_V.
\]
As $U$ is $\skewT$-invariant and $t\in U_y$ for some $y\in C_n$  by Proposition \ref{every vertical point attained}, we know that $g^n_y(t)\in [-\|u\|_\infty,\|u\|_\infty]$. Therefore, $g^n_x(t)\in [-\|u_s\|_\infty-C_V,\|u_s\|_\infty+C_V]=:V$ as required.
\end{proof}
\end{Lem}

Now we can prove our bounded distortion result in the vertical direction.  We prove this in two parts, by firstly varying the horizontal coordinate then the vertical coordinate. This is much more subtle than bounded distortion in, for example, \cite[Lemma 3]{Bedford} as we do not have uniform contraction of the heights of cylinders.

\begin{Lem}\label{bounded distortion fix t}
Let $t\in U_{C_n}$. Let $\skewT\in \mathcal{S}(X,\RR, C^{\Lip,2}_P)$ be partially hyperbolic. Then for all $x,y \in C_n$, there exists constant $D_2>0$ such that
\begin{equation}\label{bdd dist fix t eq}
D_2 \leq \frac{\prod_{j=0}^{n-1}\del h_{T^jx}(g^{j+1}_x(t))}{\prod_{j=0}^{n-1}\del h_{T^jy}(g^{j+1}_y(t))}\leq D_2.
\end{equation}
\begin{proof}
As $\log$ is smooth and $X$ is compact, it suffices to bound
\begin{align}\label{bound fix t 0}
\left|\sum_{j=0}^{n-1}\log\del h_{T^jx}(g^{j+1}_x(t))-\log\del h_{T^jy}(g^{j+1}_y(t))\right|&\leq \sum_{j=0}^{n-1} C'_g(T^jx)|g^{j+1}_x(t) -g^{j+1}_y(t)| + C_g(g^{j+1}_y(t))d(T^jx,T^jy).
\end{align}
Denoting the least Lipschitz constant of $t\mapsto\log \del g_x(t)$ by $C'_g(x)$ and the least Lipschitz constant of $x\mapsto\log\del g_x(t)$ by $C_g(t)$. Both exist as $g\in C_P^{2,\Lip}$. As $X$ is compact and by Lemma \ref{itrt forw t bdd}, $g^{j+1}_y(t)\in V$ for $j$ such that $0\leq j\leq n-1$, the Lipschitz constants $C'_g(x)$ and $C_g(t)$ are bounded independently of $C_n$. We bound
\begin{align}\label{cont bound fix t}
\nonumber |g^{j+1}_x(t) -g^{j+1}_y(t)| &\leq \sum_{i=0}^{j}|g_{T^{j}x}\dots g_{T^{i+1}x}g_{T^ix}g_{T^{i-1}y}\dots g_y(t)- g_{T^{j}x}\dots g_{T^{i+1}x}g_{T^iy}g_{T^{i-1}y}\dots g_y(t)|\\
\nonumber &\leq\sum_{i=0}^{j}\|\del g_{T^{j}x}\|_\infty\dots \|\del g_{T^{i+1}x}\|_\infty|g_{T^ix}g_{T^{i-1}y}\dots g_y(t) - g_{T^iy}g_{T^{i-1}y}\dots g_y(t)|\\
&\leq \sum_{i=0}^{j}C_g(g^i_yt)\|\del g\|_\infty^{j-i} d(T^ix, T^iy).
\end{align}
Again, $g^i_y(t)\in V$ for $i\leq n$, so the constant $C_g(g^i_y(t))\leq C_V$ where $C_V>0$ is independent of the choice of cylinder. As $x,y\in C_n$, we have $T^ix, T^iy\in C_{n-i}$ for some cylinder of rank $n-i$.  By Proposition~\ref{bounded distortion}, there exists $D_0>0$ such that
\[
d(T^ix, T^iy)\leq D_0 (\m|T|)^{-(n-i)} = D_0(\m|T'|)^{-(n-j)}(\m|T'|)^{-(j-i)}.
\]
By partial hyperbolicity, \eqref{cont bound fix t} is bounded by
\[
\sum_{i=0}^{j}C_VD_0\ka^{-(j-i)}(\m|T'|)^{-(n-j)}\leq C_1(\m|T'|)^{-(n-j)},
\]
for some constant $C_1>0$. Substituting this into \eqref{bound fix t 0} and using proposition~\ref{bounded distortion}, we obtain
\begin{align*}
\sum_{j=0}^{n-1} C'_g(T^jx)|g^{j+1}_x(t) -g^{j+1}_y(t)| + C_Vd(T^jx,T^jy)&\leq C_2\sum_{j=0}^{n-1}(\m|T'|)^{-(n-j)}.
\end{align*}
As $(\m|T'|)^{-1}<1$, the sum converges as required.
\end{proof}
\end{Lem}

We now perturb the vertical coordinate.
\begin{Lem}\label{bounded distortion vertical}
Let $x\in C_n$. Let $\skewT\in \mathcal{S}(X,\RR, C^{\Lip,2}_P)$ be partially hyperbolic. Then for all $s,t\in U_{C_n}$, there exists constant $D_3>0$ such that
\begin{equation}\label{to bound for bdd dist fibre}
D_3\leq\frac{\prod_{j=0}^{n-1}\del h_{T^jx}(g^{j+1}_x(t))}{\prod_{j=0}^{n-1}\del h_{T^jx}(g^{j+1}_x(s))}\leq D_3.
\end{equation}
\begin{proof}
If $x \in R_P$ then, as $\del g_p(t)=1$ for $p\in P$ and all $t\in
\RR$, we can ignore the terms of the product when $T^jx\in P$. By
Lemma \ref{itrt forw t bdd}, $g^{j+1}_x(t)$ and $g^{j+1}_x(s)\in
V$. As $\log$ is smooth on the compact set $X$, it suffices to bound
\begin{align}\label{bounded height fix t eq0}
\sum_{j=0}^{n-1}|\log\del h_{T^jx}(g^{j+1}_x(t))-\log\del h_{T^jx}(g^{j+1}_x(s))|\leq \sum_{j=0}^{n-1}\sup_{r\in V}|\del(\log\del h_{T^jx}(r))||g^{j+1}_x(t)-g^{j+1}_x(s)|,
\end{align}
as $h_x:\RR\to \RR$ is $C^2$. Denote $t' = g^n_x(t)\in V$ and $s' = g^n_x(s)\in V$. As $h_p$ is the identity for $p\in P$, we have $\del(\log\del h_p(t))=0$ for all $t\in\RR$, so there exists $r'\in V$ such that \eqref{bounded height fix t eq0} can be written as
\begin{align}\label{bounded height fix t eq2}
&\nonumber\sum_{j=0}^{n-1}\sup_{r\in V}\left|\del(\log\del h_{T^jx})(r)\right||h_{T^{j+1}x}\dots h_{T^{n-1}x}(t')-h_{T^{j+1}x}\dots h_{T^{n-1}x}(s')|\\
&\hspace{2cm}=\sum_{j=0}^{n-1}|\del(\log\del h_{T^jx}(r'))-\del(\log\del h_p(r'))||h_{T^{j+1}x}\dots h_{T^{n-1}x}(t')-h_{T^{j+1}x}\dots h_{T^{n-1}x}(s')|.
\end{align}
Recall, by assumption, that $x\mapsto\del^2 h_{T^jx}(t)$ is
Lipschitz continuous. As $r'\in V$, there is a constant $C_V$ depending
on $V$ such that \eqref{bounded height fix t eq2} is bounded by
\begin{equation}\label{bounded height fix t eq3}
\sum_{j=0}^{n-1}C_V d(T^jx, p)|h_{T^{j+1}x}\dots h_{T^{n-1}x}(t')-h_{T^{j+1}x}\dots h_{T^{n-1}x}(s')|.
\end{equation}
By the Mean Value Theorem, \eqref{bounded height fix t eq3} is bounded by
\begin{equation}\label{bounded height fix t eq4}
\sum_{j=0}^{n-1}C_V d(T^jx, p)\|\del (h_{T^{j+1}x}\dots h_{T^{n-1}x})\|_\infty|t'-s'|
\end{equation}
where $\|\cdot\|_\infty$ denotes the uniform norm in the fibre direction.
As $|t'-s'|\leq \diam V$ we can incorporate it in the constant $C_V$
and redefine it as $C_V\diam V$.  Consider the orbit segment $x,\dots,
T^{n-1}x$. Let $\de>0$ be sufficiently small so that Lemma \ref{aux
  lem bad set} holds. Suppose that $K$ elements of the orbit $x,\dots,
T^{n-1}x$ are contained in $X\setminus B_\de(P)=G$.

Denote each visit to $G$ by $T^{j_k}x\in G$. If $x\not\in G$, we let $n_{-1}=0$ and $n_0$ is the first visit to $G$, and if $x\in G$ then $n_0=0$. Thus, we can write the sum \eqref{bounded height fix t eq4} as
\begin{equation}\label{bounded height fix t eq5}
\sum_{k=-1}^{K-1}\left(\sum_{j=i_k}^{i_{k+1}-1}C_V d(T^jx, p)\|\del (h_{T^{j+1}x}\dots h_{T^{n-1}x})\|_\infty\right),
\end{equation}
where $T^jx\in B_\ep(P)$ for $i_k<j<i_{k+1}$. Consider each summand of \eqref{bounded height fix t eq5} individually. For $i_k\leq j<i_{k+1}$ there are at least $K-k-1$ visits  to $G$ in the orbit segment $T^jx\dots T^{n-1}x$. Hence, by condition \eqref{fast contr diff} the $k^{th}$ term of \eqref{bounded height fix t eq5} is bounded by
\begin{equation}
\sum_{j=i_k}^{i_{k+1}-1}C_V d(T^jx, p)\|\del (h_{T^{j+1}x}\dots h_{T^{n-1}x})\|_\infty \leq C_\mathcal{S}\la_\de^{K-k-1}\sum_{j=i_k}^{i_{k+1}-1}C_V d(T^jx, p).
\end{equation}
Finally, as $T^{i_k+1}x\dots T^{i_{k+1}-1}x\in B_\ep(P)$, by Lemma \ref{aux lem bad set} and bounded distortion, the sum converges to, say $C_1>0$. So, if $x\in G$
\begin{align}\label{bounded height fix t eq7}
\nonumber \sum_{j=0}^{n-1}|\log\del h_{T^jx}(g^{j+1}_x(t))-\log\del h_{T^jx}(g^{j+1}_x(s))|&\leq \sum_{k=0}^{K-1}C_\mathcal{S}C_1\la_\de^{K-k-1}\leq C_2.
\end{align}
If $x\not\in G$, by Lemma \ref{aux lem bad set}, the extra $k=-1$ term is bounded by $\sum_{j=0}^{i_0-1}C_1d(T^jx, p) \leq C_{3}$. Letting $D_2=e^{C_{2}+ C_{3}}$ we have completed the proof.
\end{proof}
\end{Lem}

This gives us the vertical bounded distortion estimate that we need. By these results and Lemma \ref{bound on height} we have the following.
\begin{Lem}\label{bound on height as want}
Let $\skewT\in \mathcal{S}(X,\RR, C^{\Lip,2}_P)$ be partially hyperbolic. There exists constant $D>1$ such that for any $x\in C_n$,
\[
D^{-1}\mathcal{D} h^{n}(x)\leq \mathop{\textnormal{Height}}(U_{C_n})\leq D\mathcal{D} h^{n}(x).
\]
\begin{proof}
By Lemma
  \ref{bound on height}, there exists $(y,t)\in C_n\times U_{C_n}$
  satisfying the upper bound and $(z,r)\in C_n\times U_{C_n}$
  satisfying the lower bound. By Lemmas \ref{bounded distortion}, \ref{bounded distortion fix t} and \ref{bounded
    distortion vertical} we have, for any $(x,s)\in C_n\times U_{C_n}$,
\[
(D_1D_2D_3)^{-1}\prod_{j=0}^{n-1}\del h_{T^jx}(g^{j+1}_x(s))\leq \Height(U_{C_n})\leq D_1D_2D_3\prod_{j=0}^{n-1}\del h_{T^jx}(g^{j+1}_x(s)).
\]
Let $D=D_1D_2D_3$. In particular, choosing $s=\sup(U_x)$, the result follows.
\end{proof}
\end{Lem}

\subsection{Proof of Theorem \ref{box dim thm}}
We now calculate the box dimension of the graph of $u$ under the partial hyperbolicity hypothesis.

\begin{Def}\label{pressure}
Let $\UU_n$ be the set of all cylinders of rank $n$. Let $C_\ell\in \UU_\ell$. Let $\phi:X\to\RR$ be a function. Define
\[
Z(\phi,n) = \inf_{\substack{\UU\subset\bigcup_{\ell\geq n}\UU_\ell\\ \UU \text{ covers } X}}\sum_{C_\ell\in \UU}\exp\sup_{x\in C_\ell}\phi^\ell(x)\hspace{0.5cm}\text{and}\hspace{0.5cm}
CZ(\phi, n) = \sum_{C_n\in \UU_n}\exp\sup_{x\in C_n}\phi^n(x).
\]
We define the \emph{topological pressure} and \emph{upper capacity topological pressure} of $\phi$ over $X$ respectively as $\PP(\phi) = \lim_{n\to\infty}\frac{1}{n}\log Z(\phi, n)$ and $\overline{C\PP}(\phi) = \limsup_{n\to\infty}\frac{1}{n}\log CZ(\phi, n)$.
\end{Def}
Suppose $\phi$ is continuous.  As we are taking the pressure over a
compact, $T$-invariant set $X$, the upper and lower capacity
topological pressure and the topological pressure are equal \cite[Theorem
  11.5]{Pesin}.

\begin{Def}
Let $U$ be a non-empty bounded subset of $\RR^2$. Let $N(U,r)$ be the least number of balls of diameter $r$ needed to cover $U$. Define the lower and upper box counting dimension of $U$ as
\[
\underline{\dim}_BU = \liminf_{r\to 0}\frac{\log N(U,r)}{-\log r} \hspace{1.5cm} \overline{\dim}_BU = \limsup_{r\to 0}\frac{\log N(U,r)}{-\log r}.
\]
respectively.  When the limits coincide, the \emph{box dimension} $\dim_BU$ of $U$ is the common value.
\end{Def}
By \cite[3.1]{Falconer1} we can replace $N(U,r)$ by the maximum number of disjoint balls of diameter $r$ with centres in $U$, denoted $\tilde N(U,r)$. When covering the set $U$ by boxes, we will need to be able to accurately estimate the size of the cylinder sets. To do this we will use Moran covers.
\begin{Def}
Let $r>0$. For all $x\in X$ define the unique integer $n(x)$ such that
\begin{equation}\label{def of moran}
\left(|T'|^{n(x)}(x)\right)^{-1}>r \hspace{.5cm}\text{and}\hspace{.5cm}
r\|T'\|_\infty^{-1}<\left(|T'|^{n(x)+1}(x)\right)^{-1} \leq r.
\end{equation}
Let $C_x$ be a cylinder set containing $x$. If $x'\in C_x$ and $n(x')\leq n(x)$, then $C_x\subseteq C_{x'}$. Let $R_i$ be the largest cylinder containing $x$ such that $R_i = C_{x'}$ for some $x'\in R_i$ and $n(x')\leq n(x'')$ for all $x''\in R_i$. The collection of all such $R_i$ determines a pairwise disjoint partition of $X$ (except at the end-points of cylinders). Let $\mathcal{M}_r = \{R_i\}_{i=0}^{|\mathcal{M}_r|-1}$, denoting the number of cylinders in $\mathcal{M}_r$ by $|\mathcal{M}_r|$. We call $\mathcal{M}_r$ a \emph{Moran cover of $X$ at scale $r$ with respect to $|T'|^{-1}$}.
\end{Def}

We first determine a lower bound on the lower box dimension of $U$.
\begin{Lem}\label{lbd}
Let $\al>0$. Let $\skewT\in \mathcal{S}(X,\RR, C^{\Lip,2}_P)$ be partially hyperbolic with invariant quasi-graph $U$ not the graph of a continuous function. Let $t$ be the unique solution to the generalised Bowen equation \eqref{bowen eq}. Then $\underline{\dim}_BU \geq t.$
\begin{proof}
Let $\mathcal{M}_r = \{R_i\}$ be a Moran cover at scale $r$ with respect to $|T'|^{-1}$. Let $x\in X$ be given, then define $x_i\in R_i$ as the pre-image of $x\in X$ with least $n$ such that $T^nx_i = x$. As $R_i$ is a cylinder, such an $x_i$ exists. Denote the rank of each cylinder $R_i$ by $n(x_i)$. By Lemma \ref{bounded distortion} there exists $D_0\geq 1$ such that
\begin{align}\label{bound on width cyls}
D_0^{-1}r<D_0^{-1}\left(|T'|^{n(x_i)}(x_i)\right)^{-1}\leq |R_i|.
\end{align}

Therefore, for each element $R_i$ of $\mathcal{M}_r$, we can find at least one open interval of diameter $D_0^{-1}r$ centred in $R_i$. Let $p\in P$. There exists $y_i\in R_i$, a pre-image of $p$ such that $T^{n(x_i)}y_i=p$ and $y_i$ is the pre-image of $p$ in $R_i$ with smallest $n$ such that $T^ny_i=p$. By Lemmas \ref{lower bound on height}, \ref{bounded distortion fix t}, \ref{bounded distortion vertical} and \ref{bound on height as want}, we have
\begin{align*}
\Height(U_{R_i})&\geq |U_{y_i}| \geq \inf_{t\in U_{T^{n(x_i)}y_i}}\|\del h^{n(x_i)}_{y_i}(t)\|_\infty|U_{p}|\geq C_0^{-1}\mathcal{D} h^{n(x_i)}(y_i).
\end{align*}
for $C_0=D|U_p|>0$ where $D$ is as in Proposition \ref{bound on height as want}.

By Lemma \ref{bounded distortion} there exists $D_0>0$ such that each $y_i$ must be distance $D_0^{-1}|R_i|$ apart.  By \eqref{bound on width cyls}, each $y_i$ must be distance $D_0^{-2}r$ apart. Let $K_0^{-1}=D_0^{-2}$. We can find at least
\begin{equation}\label{max disj balls 1}
C_0^{-1}\frac{\mathcal{D} h^{n(x_i)}(y_i)}{K_0^{-1}r}-1
\end{equation}
disjoint balls of diameter $K_0^{-1}r$ with centre in $U_{y_i}$ and not intersecting any other interval $U_{y_j}$.

Therefore, by \eqref{max disj balls 1} and bounded distortion, letting $C_1=D_0C_0$ we see that
\begin{equation}\label{lower bound on balls}
C_1^{-1}\frac{\mathcal{D} h^{n(x_i)}_{x_i}(x_i)}{K_0^{-1}\left(|T'|^{n(x_i)}(x_i)\right)^{-1}}-1< C_1^{-1}\frac{\mathcal{D} h^{n(x_i)}_{x_i}(x_i)}{K_0^{-1}r}-1\leq  \tilde N(U_{R_i}, K_0^{-1}r).
\end{equation}
By partial hyperbolicity, for $n(x_i)$ sufficiently large $|\mathcal{D} h_x^{n(x_i)}(x_i)|T'|^{n(x_i)}(x_i)|>\ka^{n(x_i)}$ is large and so we can incorporate the $-1$ into the constant. Let $\underline d = \underline{\dim}_BU$. Fix $\ep>0$. By the definition of lower box dimension, for any  $\ep>0$ there exists sequence $r_\ell\to 0$ such that $\tilde N(U,r_\ell)<r_\ell^{-(\underline{d}+\ep)}.$

Let $\phi_\ep = (\log| T'| - (\underline{d}+\ep)\log| T'| + \log \mathcal{D} h)$. Let $n_\ell$ be the infimum of the rank of cylinders in $\mathcal{M}_r$, so $n_\ell = \inf\{n(x_i)\mid x_i\in R_i\}$ and notice that as $\ell\to \infty$, $n_\ell\to\infty$. As $\mathcal{M}_r$ is an example of a cover $\UU$ of $X$ for cylinders of rank at least $n_\ell$, then
\begin{align*}
Z(\phi, n_\ell) &\leq \sum_{i=0}^{|\mathcal{M}_{r_\ell}|-1}\exp\sup_{x\in R_i}\phi^{n(x_i)}(x)\leq \sum_{i=0}^{|\mathcal{M}_r|-1}\sup_{x\in R_i}\left(|T'|^{n(x_i)}(x)\right)^{-(\underline{d}+\ep)}\frac{\mathcal{D} h^{n(x_i)}(x)}{\left(|T'|^{n(x_i)}(x)\right)^{-1}}.
\end{align*}
By bounded distortion, Lemmas \ref{bounded distortion}, \ref{bounded distortion vertical} and \ref{bounded distortion fix t}, we can replace the supremum by a constant $C_2$ and sum over $x_i$. Thus,
\begin{align*}
Z(\phi, n_\ell)&\leq C_2\sum_{i=0}^{|\mathcal{M}_r|-1}\left(|T'|^{n(x_i)}(x_i)\right)^{-(\underline{d}+\ep)}\frac{\mathcal{D} h^{n(x_i)}(x_i)}{\left(|T'|^{n(x_i)}(x_i)\right)^{-1}}\leq C_3 r_\ell^{\underline{d}+\ep}r_\ell^{-(\underline{d}+\ep)}= C_3,
\end{align*}
where $C_3>0$ is independent of $r_\ell$. Hence, $\frac{1}{n_\ell}\log Z(\phi_\ep,n_\ell)\leq \frac{1}{n_\ell}\log C_3$. Letting $r_\ell\to 0$, hence $n_\ell\to\infty$, we have $\PP(\phi_\ep)\leq 0$. As pressure is monotone increasing, $t\leq \underline{d}+\ep$. As $\ep>0$ is arbitrary, $t\leq \underline{d}$.
\end{proof}
\end{Lem}

We prove that the upper box dimension is bounded above by the unique solution $t$ to the generalised Bowen equation \eqref{bowen eq}.
\begin{Lem}\label{ubd}
Let $\skewT\in \mathcal{S}(X,\RR, C^{\Lip,2}_P)$ be partially hyperbolic with invariant quasi-graph not the graph of a continuous function. Let $t$ be the unique solution to the generalised Bowen equation \eqref{bowen eq}. Then $\overline{\dim}_BU \leq
t.$
\end{Lem}
\begin{proof}
Again, let $\mathcal{M}_r = \{R_i\}$ be a Moran cover at scale $r$ with respect to $|T'|^{-1}$ and let $x_i\in R_i$ be the pre-image of rank $n(x_i)$ of a given point $x$. So,
\begin{equation*}
|R_i|\leq D_0\left(|T'|^{n(x_i)+1}(x_i)\right)^{-1}\leq D_0 r.
\end{equation*}
Therefore, there are at most $D_0+1 = C_0$ balls of diameter $r$ needed to cover each $R_i$. By Proposition \ref{bound on height as want} and applying bounded distortion, Lemmas \ref{bounded distortion fix t} and \ref{bounded distortion vertical}, the height of $U$ over $R_i$ is bounded by
\[
\mathop{\textnormal{Height}}(U_{R_i})\leq D\mathcal{D} h^{n(x_i)}(x_i).
\]
So, the number of balls of diameter $r$ needed to cover $U$ over $R_i$ is bounded by
\begin{equation}\label{number of balls upper 2}
N(U_{R_i},r)\leq C_1\frac{\mathcal{D} h^{n(x_i)}(x_i)}{\left(|T'|^{n(x_i)}(x_i)\right)^{-1}}+1,
\end{equation}
where $C_1=DC_0>0$. By partial hyperbolicity as above, we can incorporate the $1$ into our constant.

Let $\overline{d} = \overline{\dim}_BU$. Let $\ep>0$ be arbitrary. By rearranging the definition of upper box dimension, there exists a sequence $r_\ell\to 0$ such that $N(U,r_\ell)\geq r_\ell^{\ep-\overline{d}}$.

Taking logs of the definition of a Moran cover, we see that there exist constants $C_2, C_3>0$ independent of $r_\ell$ such that
\[
-C_2\log r_\ell -1 \leq n(x_i) \leq -C_3\log (\|T\|_\infty^{-1}r_\ell) +1.
\]
For sufficiently small $r_\ell>0$, $n(x_i)$ can only take $B\leq -C_3\log (\|T\|_\infty^{-1}r_\ell)$ different values.

Let $U|_{\{C_N\in \mathcal{M}_{r_\ell}\}}$ be the restriction of $U$ to cylinders in the Moran cover $\mathcal{M}_{r_\ell}$ of rank $N$. Denote the number of balls of diameter $r$ required to cover $U|_{\{C_N\in \mathcal{M}_{r_\ell}\}}$ by $N(U|_{\{C_N\in \mathcal{M}_{r_\ell}\}}, r_\ell)$. As $n(x_i)$ can only take $B$ different values, there exists $N$ such that
\begin{align*}
N(U|_{\{C_N\in \mathcal{M}_{r_\ell}\}}, r_\ell)&\geq \frac{N(U,r_\ell)}{B}\geq \frac{r_\ell^{\ep-\overline{d}}}{-C_3\log (\m(|T'|^{-1})r_\ell)}\geq r_\ell^{2\ep-\overline{d}}
\end{align*}
for sufficiently small $r_\ell$. Let $\psi_{2\ep} = (\log |T'| - (\overline{d}-2\ep)\log |T'| + \log\mathcal{D} h)$. Then, again by bounded distortion, Lemmas \ref{bounded distortion}, \ref{bounded distortion fix t} and \ref{bounded distortion vertical}, we have
\begin{align*}
\overline{CZ}(\psi_{2\ep}, N) &= \sum_{C_N \in \UU_N}\exp\sup_{x\in C_N}(\psi_{2\ep}^Nx)\geq C_4\sum_{C_N\in \mathcal{M}_{r_\ell}}\left(|T'|^N(x_i)\right)^{\overline{d}-2\ep}	\frac{\mathcal{D} h^{N}(x_i)}{\left(|T'|^N(x_i)\right)^{-1}}.
\end{align*}
Thus, $\overline{CZ}(\psi_{2\ep}, N)\geq C_4r_\ell^{\overline{d}-2\ep}r_\ell^{2\ep-\overline{d}}= C_4$
where $C_4$ is independent of $r_\ell$. As $r_\ell\to 0$, $N\to\infty$, $\overline{CP}(\psi_{2\ep})\geq  0$. As $X$ is compact, $\PP(\psi_{2\ep})=\overline{CP}(\psi_{2\ep})$. Pressure is monotone increasing and $\PP(\psi_t)=0$, so $t\geq \overline{d}+2\ep$. As $\ep$ is arbitrary, $t\geq \overline{d}$.
\end{proof}

\begin{proof}[Proof of Theorem \ref{box dim thm}]
By Lemmas \ref{lbd} and \ref{ubd}, $t\leq \underline{\dim}_BU\leq \overline\dim_BU\leq t$. As $t$ is unique \cite{MossWalkdenDim}, $t = \dim_{B}U$.
\end{proof}

\begin{Rem}
An open question is the Hausdorff dimension of the quasi-graph. There
has been much significant progress in the study of Hausdorff dimension
of (uniformly contracting) Weierstrass-type graphs, c.f. \cite{Otani},
\cite{Shen}, \cite{BBR}; by relating the natural extension of the skew
product to fat solenoidal maps using Ledrappier-Young Theory
\cite{LedrappierYoung}, \cite{Ledrappier}.  The obstruction in the setting of this paper is developing transversality results (cf.\ \cite{Tsujii}) to
quasi-graphs.
\end{Rem}

\end{document}